\documentclass[12pt,reqno]{amsart}

\usepackage{amsmath,amssymb,amsthm}
\usepackage[a4paper,margin=1.02in]{geometry}
\usepackage[colorlinks=true,citecolor=blue,linkcolor=blue,urlcolor=blue]{hyperref}
\allowdisplaybreaks
\numberwithin{equation}{section}

\newcommand{\R}{\mathbb R}
\newcommand{\C}{\mathbb C}
\newcommand{\dd}{\,\mathrm d}
\newcommand{\ind}{\operatorname{ind}}
\newcommand{\Rea}{\operatorname{Re}}
\newcommand{\Ima}{\operatorname{Im}}
\newcommand{\e}{\mathrm e}

\newtheorem{theorem}{Theorem}[section]
\newtheorem{proposition}[theorem]{Proposition}
\newtheorem{lemma}[theorem]{Lemma}
\newtheorem{openproblem}{Open problem}[section]

\title[Planar traveling waves at every subsonic speed]
{Planar Gross--Pitaevskii traveling waves at every subsonic speed}
\author{Changfeng Gui, Shanfa Lai, Guolin Qin, Juncheng Wei}

\address{(C. Gui) Department of Mathematics, Faculty of Science, University of Macau, Taipa, Macao SAR, China}
\email{changfenggui@um.edu.mo}

\address{(S. Lai) Department of Mathematics, Faculty of Science, University of Macau, Taipa, Macao SAR, China}
\email{laishanfa@amss.ac.cn}

\address{(G. Qin) State Key Laboratory of Mathematical Sciences, Academy of Mathematics and Systems Science, Chinese Academy of Sciences, Beijing 100190, P.R. China and University of Chinese Academy of Sciences, Beijing 100049,  P.R. China}
\email{qinguolin18@mails.ucas.ac.cn}

\address{(J. Wei) Department of Mathematics, The Chinese University of Hong Kong, Shatin, Hong Kong.}
\email{wei@math.cuhk.edu.hk}


\begin{document}

\begin{abstract}
For every subsonic speed $c\in(0,\sqrt2)$, we prove the existence of a finite-energy traveling wave for the planar Gross--Pitaevskii equation.  This resolves the longstanding problem of the existence of prescribed-speed traveling-wave solutions in two dimensions, explicitly stated as open by Mari\c{s} (Ann. of Math., 2013) and Bellazzini and Ruiz (Amer. J. Math., 2023).  The proof relies essentially on the energy estimate
\begin{equation*}
 E(\psi)\le C_J\bigl(I_c(\psi)+\ind(\psi)\bigr),
 \qquad c\in J,
\end{equation*}
where $E$ is the energy, $I_c$ the action at speed $c$, $\ind$ the real Morse index, $J$ is any compact interval contained in $(0,\sqrt2)$, and $C_J$ is a positive constant depending only on $J$. We also prove finite-bubble compactness, including splitting of the energy, action, potential energy, and momentum, and attainment of the action among nonconstant waves of Morse index at most one.
\end{abstract}

\maketitle
\enlargethispage{3pt}

\section{Introduction and main results}

For $\mathbf{x}=(x,y)\in\R^2$ and $t\in\R$, the Gross--Pitaevskii equation is
\begin{equation}\label{eq:1.1}
 i\partial_t\Psi+\Delta\Psi+(1-|\Psi|^2)\Psi=0
 \qquad\hbox{in }\R^2\times\R,
\end{equation}
where $i^2=-1$ and $\Psi$ is complex-valued. Pitaevskii and Gross introduced \eqref{eq:1.1} as a model for superfluidity \cite{Pitaevskii,Gross}. The equation is also a defocusing nonlinear Schr\"odinger equation with a nonzero background. Its vortices and coherent structures also arise in Bose--Einstein condensates and nonlinear optics \cite{Berloff,BethuelGravejatSautSurvey}.

In this paper, we study traveling waves for \eqref{eq:1.1}. A traveling wave moving in the negative $x$ direction with speed $c>0$ has the form $\Psi(\mathbf{x},t)=\psi(x+ct,y)$. Substituting in \eqref{eq:1.1}, we find
\begin{equation}\label{eq:1.2}
 ic\partial_x\psi+\Delta\psi+(1-|\psi|^2)\psi=0
 \qquad\hbox{in }\R^2.
\end{equation}
The corresponding Ginzburg--Landau energy is defined by
\begin{equation}\label{eq:1.3}
 E(\psi)=\frac12\int_{\R^2}|\nabla\psi|^2\dd\mathbf{x}
 +\frac14\int_{\R^2}(1-|\psi|^2)^2\dd\mathbf{x}.
\end{equation}
For every finite-energy solution $\psi$ of \eqref{eq:1.2}, \cite[Theorem~9]{BethuelGravejatSautSurvey} yields a constant $\lambda_\infty\in\C$, with $|\lambda_\infty|=1$, such that
\begin{equation*}
 \lim_{R\to\infty}\ \sup_{|\mathbf{x}|\ge R}
 |\psi(\mathbf{x})-\lambda_\infty|=0.
\end{equation*}
Replacing $\psi$ by $\overline{\lambda_\infty}\psi$, which leaves \eqref{eq:1.2} and the energy unchanged, we henceforth assume that
\begin{equation}\label{eq: normalize}
    \lim_{|\mathbf{x}|\to\infty}\psi(\mathbf{x})=1.
\end{equation}
For a solution satisfying \eqref{eq: normalize}, the decay estimates \cite[Theorem~11 and Proposition~33]{GravejatDecay} show that the momentum integral below is absolutely convergent. In the same spirit as \cite[Section~2]{BellazziniRuiz}, we define the momentum and the action at speed $c$ by
\begin{equation*}
 P(\psi)=-\int_{\R^2}\partial_x(\Ima\psi)(\Rea\psi-1)
 \dd\mathbf{x},
 \qquad I_c(\psi)=E(\psi)-cP(\psi).
\end{equation*}

The sound speed $c=\sqrt2$ can be derived by linearization at the constant solution. Jones and Roberts, and later Jones, Putterman, and Roberts, used formal calculations and numerical continuation to predict nonconstant waves throughout the subsonic interval \cite{JonesRoberts,JonesPuttermanRoberts}. This prediction is now part of the program of Jones, Putterman, and Roberts discussed in \cite{BethuelGravejatSautSurvey}. Gravejat proved that nonconstant finite-energy waves do not exist for $c>\sqrt2$, or for $c=\sqrt2$ in dimension two \cite{GravejatSupersonic,GravejatSonic}. The remaining prescribed-speed question is therefore the following.

\begin{openproblem}\label{prob:1.1}
Does \eqref{eq:1.2} have a nonconstant finite-energy traveling solution for every $c\in(0,\sqrt2)$?
\end{openproblem}

There has been much progress on Open problem~\ref{prob:1.1}. B\'ethuel and Saut \cite{BethuelSaut} gave the first rigorous construction in the plane. They proved the existence of a nonconstant finite-energy traveling wave for every sufficiently small speed $c>0$ and described the corresponding vortex pair, whose two vortices have opposite degrees and are separated by a distance of order $c^{-1}$. B\'ethuel, Gravejat, and Saut \cite{BethuelGravejatSaut} subsequently minimized the energy at fixed positive momentum and thereby constructed a planar traveling wave for every prescribed momentum. Chiron and Mari\c{s} \cite{ChironMaris} extended this fixed-momentum method to a broad class of nonlinear Schr\"odinger equations with nonnegative potential and proved the orbital stability of the resulting family of minimizers.

These variational constructions yield traveling waves whose speeds come arbitrarily close to both endpoints of the subsonic interval. Their speed, however, is obtained as a Lagrange multiplier and is not prescribed in advance. Mari\c{s} \cite[pp.~109--110]{Maris} explicitly pointed out that the set of speeds obtained in this way contains values arbitrarily close to both $0$ and $\sqrt2$, but that it was not known whether it covers the whole subsonic interval $(0, \sqrt 2)$.

For small speed, in a series of works, Chiron and Pacherie \cite{ChironPacherieBranch,ChironPacherieCoercivity, ChironPacherieUniqueness} constructed a smooth two-vortex branch and proved its coercivity and uniqueness properties. Related nonlinear coercivity and orbital stability for the degree-one, zero-speed Ginzburg--Landau vortex were proved by Gravejat, Pacherie, and Smets \cite{GravejatPacherieSmets}. Near $c=\sqrt2$, the relevant long-wave profile is the KP-I lump, and rigorous results include \cite{BethuelGravejatSautKP,ChironRarefaction,LiuWangWeiYang}. Numerical continuation exhibits several branches across the full subsonic interval \cite{ChironScheid}.

In higher dimensions $n\geq 3$, Mari\c{s} \cite{Maris} proved existence at every subsonic speed. His Pohozaev minimization uses the transverse scaling $u(x,y)\mapsto u(x,\sigma y)$. In dimension two the associated constrained infimum is zero and is not attained, so that argument does not apply. Mari\c{s} \cite[p.~119]{Maris} states explicitly that existence for every $c\in(0,\sqrt2)$ remains open in the plane. Open problem~\ref{prob:1.1} is also listed in \cite[Section~1.1.3]{PacherieSurvey} and \cite[Section~1.2.2]{PacherieRecent}.

The closest planar result is due to Bellazzini and Ruiz \cite{BellazziniRuiz}. They first solve the equation on expanding slabs. Struwe's monotonicity argument results in uniformly bounded mountain-pass levels for almost every speed \cite{Struwe}, while the second-order min--max argument of Fang and Ghoussoub \cite{FangGhoussoub} provides a Morse-index bound. The resulting whole-plane statement \cite[Theorem~1.1]{BellazziniRuiz} is that, for almost every $c\in(0,\sqrt2)$, there is a nonconstant finite-energy wave $\psi_c$. On each fixed compact subinterval of $(0,\sqrt2)$ these waves satisfy
\begin{equation*}
 0<I_c(\psi_c)\le A,
 \qquad \ind(\psi_c)\le1.
\end{equation*}
Here $\ind(\psi_c)$ is the Morse index defined below in \eqref{eq:1.5}. To obtain a wave at a missing speed $c$, Bellazzini and Ruiz propose choosing admissible speeds $c_n\to c$ and passing to a limit. Their nonvanishing result produces, after translation, a local limit whose modulus at the origin is not one \cite[Proposition~6.1]{BellazziniRuiz}. The displayed estimates do not, however, bound $E(\psi_{c_n})$. Thus local convergence alone does not exclude an unbounded amount of energy escaping to infinity.

The main difficulty is to prove this energy estimate in the plane. Bellazzini and Ruiz \cite[Theorem~1.3]{BellazziniRuiz} obtained it when the waves have no zeros, or when all zeros remain in one fixed ball with a uniform nonvanishing bound on its boundary. However, it remains unknown how to show the desired energy bound without any additional assumption. Hence Open problem~\ref{prob:1.1} still remains open.

Several other qualitative results apply once a finite-energy wave is already known. Gravejat proved uniform convergence and algebraic decay of the wave and its derivatives, and later obtained sharper asymptotic formulas \cite{GravejatDecay,GravejatAsymptotics,GravejatFirstOrder}. There are many other related works. Gui's Hamiltonian identities \cite{GuiHamiltonian} and the argument of Wei and Yao \cite{WeiYao} imply asymptotic axisymmetry and vanishing transverse momentum. G\'erard and Zhang  \cite{GerardZhang} proved nonlinear orbital stability of the one-dimensional zero-speed black soliton. In related dynamical or singular limits, Lin and Xin derived the Kirchhoff point-vortex law from nonlinear Schr\"odinger dynamics \cite{LinXin}. Serfaty studied vortex branches and critical rotation for a Gross--Pitaevskii energy \cite{SerfatyRotating}, and later established a many-vortex mean-field limit to incompressible Euler \cite{SerfatyMeanField}. Lin and Wei constructed superflows past an obstacle whose blow-up profiles are whole-plane traveling waves \cite{LinWeiObstacle}. These results use information about an existing wave, a vortex scale, or an asymptotic regime. We also refer interested readers to \cite{GravejatSmets,KochLiao,KochLiaoLowRegularity} for one-dimensional stability and conserved energies, \cite{LiuWei} for small-speed multivortex waves, \cite{MartinezSanchezRuiz,DeLaireGravejatSmets} for periodic-domain waves, and \cite{ChironHigherDimensional,BethuelOrlandiSmets,AoHuangLiuWei, ChironHelices,DavilaDelPinoMedinaRodiac} for higher-dimensional waves, rings, or helices. Related Euler and generalized surface quasi-geostrophic constructions based on rearrangements or desingularization were obtained in \cite{BurtonVortexPairs,SmetsVanSchaftingen,CaoLaiZhan,CaoLaiQin}.

In this paper, we give a complete positive answer to Open problem~\ref{prob:1.1}. By using the Morse index, we prove the missing energy estimate. It applies to every finite-energy solution of finite index and makes no assumption on the number, degrees, or locations of its zeros. Combining our key energy estimate with the almost-every-speed solutions of Bellazzini and Ruiz \cite{BellazziniRuiz} answers Open problem~\ref{prob:1.1}.

\subsection{Main results}

Complex function spaces are regarded as real spaces, and $\langle z,w\rangle=\Rea(z\overline w)$ is the real scalar product on $\C$. The second variation of the action at a solution is
\begin{equation}\label{eq:1.4}
 Q_{\psi,c}(\varphi)=\int_{\R^2}\left(
 |\nabla\varphi|^2-c\langle\varphi,i\partial_x\varphi\rangle
 -(1-|\psi|^2)|\varphi|^2+2\langle\varphi,\psi\rangle^2
 \right)\dd\mathbf{x}.
\end{equation}
Its Morse index is
\begin{equation}\label{eq:1.5}
 \ind(\psi)=\sup\left\{\dim_{\R}Y:
 \begin{array}{l}
 Y\subset C_c^\infty(\R^2,\C)\text{ is a real vector space},\\
 Q_{\psi,c}(\varphi)<0\text{ for every }\varphi\in Y\setminus\{0\}
 \end{array}\right\}.
\end{equation}

Our main energy estimate is stated as follows.
\begin{theorem}
\label{thm:1.1}
For every compact interval $J\subset(0,\sqrt2)$, there is a constant $C_J<\infty$ depending only on $J$ such that every finite-energy solution $\psi$ of \eqref{eq:1.2}, with $c\in J$, satisfies
\begin{equation}\label{eq:1.6}
 E(\psi)\le C_J\bigl(I_c(\psi)+\ind(\psi)\bigr).
\end{equation}
\end{theorem}

Combining Theorem~\ref{thm:1.1} with \cite[Theorem~1.1]{BellazziniRuiz}, we obtain traveling waves for the whole subsonic region. This therefore provides a complete answer to Open problem~\ref{prob:1.1}.

\begin{theorem}\label{thm:1.2}
For every $c\in(0,\sqrt2)$, equation \eqref{eq:1.2} has a nonconstant finite-energy solution $\psi$ with $\ind(\psi)\le1$.
\end{theorem}

The estimate in Theorem~\ref{thm:1.1} can also be used to prove finite-bubble compactness. Let
\begin{equation*}
 e(\psi)=\frac12|\nabla\psi|^2+\frac14(1-|\psi|^2)^2.
\end{equation*}
Thus $E(\psi)=\int_{\R^2}e(\psi)\dd\mathbf{x}$. As usual, $\mathbb S^1=\{z\in\C:|z|=1\}$ and $B(\mathbf{a},R)$ is the open ball of radius $R$ centered at $\mathbf{a}\in\R^2$. We abbreviate $B(\mathbf{0},R)$ to $B_R$.

B\'ethuel, Gravejat, and Saut \cite[Theorem~5.1 and Section~7.3]{BethuelGravejatSaut} proved a decomposition into finitely many profiles with energy and momentum splitting, while under \eqref{eq:1.7} we obtain the decomposition in Theorem~\ref{thm:1.3} together with the Morse index inequality \eqref{eq:1.15}.

\begin{theorem}
\label{thm:1.3}
Let $J\subset(0,\sqrt2)$ be a compact interval, let $c_n\in J$ satisfy $c_n\to c\in J$, and let $\psi_n$ be normalized finite-energy solutions of \eqref{eq:1.2} at speed $c_n$ satisfying \eqref{eq: normalize}. Assume
\begin{equation}\label{eq:1.7}
 \sup_n\bigl(I_{c_n}(\psi_n)+\ind(\psi_n)\bigr)<\infty.
\end{equation}
After passing to a subsequence, there are an integer $K\ge0$, normalized nonconstant finite-energy solutions $\psi^1,\ldots,\psi^K$ at speed $c$, centers $\mathbf{a}_n^1,\ldots,\mathbf{a}_n^K\in\R^2$, phases $\alpha_n^1,\ldots,\alpha_n^K\in\mathbb S^1$, and radii $R_n\to\infty$ such that
\begin{align}
 |\mathbf{a}_n^k-\mathbf{a}_n^\ell|&\longrightarrow\infty
 &&(k\ne\ell),\label{eq:1.8}\\
 \overline{\alpha_n^\ell}\,
 \psi_n(\mathbf{a}_n^\ell+\,\cdot)&\longrightarrow\psi^\ell
 &&\text{locally together with all derivatives},
 \label{eq:1.9}\\
 \int_{\R^2\setminus\bigcup_{\ell=1}^K B(\mathbf{a}_n^\ell,R_n)}
 e(\psi_n)\dd\mathbf{x}&\longrightarrow0.\label{eq:1.10}
\end{align}
When $K\ge2$, the radii may be chosen so that
\begin{equation*}
 \frac{R_n}{\min_{k\ne\ell}|\mathbf{a}_n^k-\mathbf{a}_n^\ell|}
 \longrightarrow0.
\end{equation*}
The energy, action, potential energy, and momentum satisfy
\begin{align}
 E(\psi_n)&\longrightarrow\sum_{\ell=1}^K E(\psi^\ell),
 \label{eq:1.11}\\
 I_{c_n}(\psi_n)&\longrightarrow
 \sum_{\ell=1}^K I_c(\psi^\ell),
 \label{eq:1.12}\\
 \int_{\R^2}(1-|\psi_n|^2)^2\dd\mathbf{x}&\longrightarrow
 \sum_{\ell=1}^K\int_{\R^2}(1-|\psi^\ell|^2)^2\dd\mathbf{x},
 \label{eq:1.13}\\
 P(\psi_n)&\longrightarrow\sum_{\ell=1}^K P(\psi^\ell).
 \label{eq:1.14}
\end{align}
The Morse indices satisfy
\begin{equation}\label{eq:1.15}
 \sum_{\ell=1}^K\ind(\psi^\ell)
 \le\liminf_{n\to\infty}\ind(\psi_n).
\end{equation}
The case $K=0$ is allowed, in which case $E(\psi_n)\to0$.
\end{theorem}

Theorem~\ref{thm:1.2} provides at least one nonconstant finite-energy wave at every subsonic speed. We now strengthen this existence result. Among the waves whose Morse index is at most one, which is the class produced in the proof of Theorem~\ref{thm:1.2}, we shall find a wave with the smallest possible action. For $c\in(0,\sqrt2)$ let
\begin{equation}\label{eq:1.16}
 \beta_1(c)=\inf\left\{I_c(\psi):
 \begin{array}{l}
 \psi\text{ is a nonconstant finite-energy solution of \eqref{eq:1.2},}\\[-2pt]
 \ind(\psi)\le1
 \end{array}\right\}.
\end{equation}
We prove that this infimum is positive and attained. We also prove the corresponding compactness result.

\begin{theorem}
\label{thm:1.4}
For every $c\in(0,\sqrt2)$, the number $\beta_1(c)$ is attained. For every compact interval $J\subset(0,\sqrt2)$ there are numbers $0<b_J\le A_J<\infty$ such that
\begin{equation}\label{eq:1.17}
 b_J\le\beta_1(c)\le A_J\qquad(c\in J).
\end{equation}
Every corresponding minimizer $\psi_c$ obeys the locally uniform energy bound
\begin{equation}\label{eq:1.18}
 E(\psi_c)\le C_J(A_J+1)\qquad(c\in J).
\end{equation}
Every minimizing sequence $(\psi_n)$ for $\beta_1(c)$ is precompact modulo translations and constant phases. Namely, after passing to a subsequence, there are $\mathbf{a}_n\in\R^2$, $\alpha_n\in\mathbb S^1$, and a minimizer $\psi_c$ with $\ind(\psi_c)\le1$ such that
\begin{equation}\label{eq:1.19}
\begin{gathered}
 \overline{\alpha_n}\psi_n(\mathbf{a}_n+\,\cdot)\longrightarrow\psi_c
 \quad\hbox{locally together with all derivatives},\\
 \left\|\nabla\bigl(\overline{\alpha_n}\psi_n(\mathbf{a}_n+\,\cdot)-\psi_c\bigr)
 \right\|_{L^2}\longrightarrow0,\\
 \left\||\psi_n(\mathbf{a}_n+\,\cdot)|^2-|\psi_c|^2\right\|_{L^2}
 \longrightarrow0.
\end{gathered}
\end{equation}
The function $c\mapsto\beta_1(c)$ is lower semicontinuous.
\end{theorem}

\subsection{Proof strategy}\label{subsec:proof-strategy}
We first explain the proof of Theorem~\ref{thm:1.1}. For a solution $\psi$, set
\begin{equation*}
 F=|\partial_x\psi|^2+\frac12(1-|\psi|^2)^2,
 \qquad
 D=|\partial_x\psi|^2-\frac12(1-|\psi|^2)^2.
\end{equation*}
In view of \eqref{eq:3.2} and \eqref{eq:3.6},
\begin{equation*}
 I_c(\psi)=\int_{\R^2}|\partial_y\psi|^2\dd\mathbf{x}
 =\int_{\R^2}D\dd\mathbf{x},
 \qquad
 E(\psi)=\frac12\int_{\R^2}
 \bigl(F+|\partial_y\psi|^2\bigr)\dd\mathbf{x}.
\end{equation*}
Thus, to prove Theorem~\ref{thm:1.1}, it remains to estimate $\int_{\R^2}F$.

We first reduce this problem into one-dimensional. For every fixed $y\in \R$, we define the slice $q_y(x):=\psi(x,y)$ and find its equation \eqref{eq:2.2}. Next, we focus on this ODE problem. For a one-dimensional function $q$, let $F_q$, $D_q$, and $R_c(q)$ be defined by \eqref{eq:2.3} and \eqref{eq:2.4}. The key ingredient is the classification of stable solutions for homogeneous equation of \eqref{eq:2.2}. In Proposition~\ref{prop:2.2}, we show that every bounded exact solution outside the family \eqref{eq:2.14} has a compactly supported negative test. With this rigidity result, we divide the real line into cells of a fixed length. It follows from Lemma~\ref{lem:2.9} that on every cell the residual is large, $D_q$ controls $F_q$, a normalized negative test exists, or the energy grows by a fixed factor on a neighboring cell. In the last case, we repeat the choice of the neighboring cell. For $F_q\in L^1(\R)$, this process stops after finitely many steps. The geometric-series estimate \eqref{eq:2.75} then bounds the total $F_q$ integral over cells assigned to the fourth alternative by the corresponding integral over cells assigned to the first three alternatives. Summing over all cells, we obtain \eqref{eq:2.73}.

For the slices $q_y(x)=\psi(x,y)$, equation \eqref{eq:2.5} and Lemma~\ref{lem:3.1} show that the residual is controlled after integration over $y$:
\begin{equation*}
 \int_\R\|R_c(q_y)\|_{L^2(\R)}^2\dd y
 =\int_{\R^2}|\partial_{yy}\psi|^2\dd\mathbf{x}
 \le C_JI_c(\psi).
\end{equation*} We next estimate the number of negative cells. If the transverse energy on a nearby rectangle is small, Lemma~\ref{lem:3.2} changes a one-dimensional negative test into a two-dimensional negative test. Tests on disjoint rectangles are counted by $\ind(\psi)$. On the other hand, the rectangles with large transverse energy are controlled by $I_c(\psi)$. Combining these two cases, Proposition~\ref{prop:3.3} proves
\begin{equation*}
 \int_{\R}\#\{\text{negative cells on the slice }y\}\dd y
 \le C_J\bigl(I_c(\psi)+\ind(\psi)\bigr).
\end{equation*}
Now, we integrate \eqref{eq:2.73} with respect to $y$. By Lemma~\ref{lem:3.1}, Proposition~\ref{prop:3.3}, and \eqref{eq:3.2}, we derive \eqref{eq:3.30}. Combining it with \eqref{eq:3.6}, we prove Theorem~\ref{thm:1.1}.

For Theorem~\ref{thm:1.2}, let $c_n\to c$ be the speeds obtained by Bellazzini and Ruiz. In view of Theorem~\ref{thm:1.1}, the corresponding solutions have uniformly bounded energy. By \cite[Proposition~6.1]{BellazziniRuiz} and Fatou's lemma, after translations we find a finite-energy local limit $u$ with $|u(0)|\ne1$. Every constant finite-energy solution has modulus one. Thus $u$ is nonconstant. It follows from \eqref{eq:3.32} that $\ind(u)\le1$.

In view of Theorem~\ref{thm:1.1}, assumption \eqref{eq:1.7} yields a uniform energy bound. By Lemma~\ref{lem:4.2}, we can only extract finitely many nonconstant profiles. We next use Lemma~\ref{lem:4.3} to show that no energy is left outside these profiles. This proves the splitting statements. By translating the negative tests of the profiles, we derive \eqref{eq:1.15}. For Theorem~\ref{thm:1.4}, we first exclude the case that the action tends to zero. Otherwise, we obtain a nonconstant finite-energy limit independent of $y$, which is impossible. Applying Theorem~\ref{thm:1.3} to minimizing sequences, we establish attainment and precompactness. Applying the same argument when $c_n\to c$, we infer lower semicontinuity.

The rest of this paper is organized as follows. In Section~2, we prove the one-dimensional estimate. In Section~3, we use this estimate to prove Theorems~\ref{thm:1.1} and~\ref{thm:1.2}. Finally, in Section~4, we prove Theorems~\ref{thm:1.3} and~\ref{thm:1.4}.

\section{The one-dimensional equation }

Let $\psi$ solve the two-dimensional equation \eqref{eq:1.2}. For a fixed number $y\in\R$, define the function of one variable
\begin{equation}\label{eq:2.1}
 q_y(x):=\psi(x,y).
\end{equation}
By substituting \eqref{eq:2.1} into \eqref{eq:1.2}, we derive the exact equation for $q_y$,
\begin{equation}\label{eq:2.2}
 q_y''+icq_y'+(1-|q_y|^2)q_y
 =-\partial_{yy}\psi(\,\cdot\,,y).
\end{equation}
We first study \eqref{eq:2.2} when its right-hand side is zero. We next prove an estimate that includes the squared $L^2$ norm of the right-hand side.

\subsection{The exact one-dimensional equation}

For a twice differentiable function $q:\R\to\C$, denote the left-hand side of \eqref{eq:2.2} by
\begin{equation}\label{eq:2.3}
 R_c(q)=q''+icq'+(1-|q|^2)q,
\end{equation}
and define the two real-valued functions
\begin{equation}\label{eq:2.4}
 F_q=|q'|^2+\frac12(1-|q|^2)^2,
 \qquad
 D_q=|q'|^2-\frac12(1-|q|^2)^2.
\end{equation}
Here $F_q\ge0$, while $D_q$ may change sign. It follows from equations \eqref{eq:2.2} and \eqref{eq:2.3} that
\begin{equation}\label{eq:2.5}
 R_c(q_y)(x)=-\partial_{yy}\psi(x,y).
\end{equation}
On the real Hilbert space $H^1(\R,\C)$, the one-dimensional quadratic form along $q$ is
\begin{equation}\label{eq:2.6}
 Q_{q,c}(\phi)=\int_\R\left(
 |\phi'|^2-c\langle\phi,i\phi'\rangle
 -(1-|q|^2)|\phi|^2+2\langle\phi,q\rangle^2
 \right)\dd x.
\end{equation}
When $I\subset\R$ is an interval, let $H_0^1(I,\C)$ denote the Sobolev space with zero trace at the endpoints of $I$. All complex Sobolev spaces are regarded as real spaces, and $Q_{q,c}^{I}$ denotes the same integral as \eqref{eq:2.6} over $I$, with form domain $H_0^1(I,\C)$.

We use the change of variables
\begin{equation*}
 \Phi(x)=\e^{icx/2}q(x),
 \qquad a=1+\frac{c^2}{4},
\end{equation*}
which has derivatives
\begin{equation*}
 \Phi'=\e^{icx/2}\left(q'+\frac{ic}{2}q\right),\qquad
 \Phi''=\e^{icx/2}\left(q''+icq'-\frac{c^2}{4}q\right).
\end{equation*}
Consequently,
\begin{equation*}
 \Phi''+(a-|\Phi|^2)\Phi
 =\e^{icx/2}R_c(q).
\end{equation*}
Thus the exact equation $R_c(q)=0$ becomes
\begin{equation}\label{eq:2.7}
 \Phi''+(a-|\Phi|^2)\Phi=0.
\end{equation}
If $\zeta=\e^{icx/2}\phi$, we derive by direct substitution in \eqref{eq:2.6}
\begin{equation}\label{eq:2.8}
 Q_{q,c}(\e^{-icx/2}\zeta)
 =\int_\R\left(|\zeta'|^2-(a-|\Phi|^2)|\zeta|^2
 +2\langle\Phi,\zeta\rangle^2\right)\dd x.
\end{equation}

\begin{lemma}\label{lem:2.1}
Let $I_1\subset I_2$ be distinct bounded open intervals. Here $\overline I_2$ denotes the closure of $I_2$, $C(\overline I_2,\R)$ denotes the space of continuous real-valued functions on $\overline I_2$, and $H_0^1(I)$ denotes the real Sobolev space of functions with zero trace at the endpoints of $I$. For $V\in C(\overline I_2,\R)$, define
\begin{equation}\label{eq:2.9}
 \lambda_1(V,I):=
 \inf_{0\ne u\in H_0^1(I)}
 \frac{\displaystyle\int_I(|u'|^2+Vu^2)\dd x}
      {\displaystyle\int_Iu^2\dd x}.
\end{equation}
If $u_1\in H_0^1(I_1)$ is strictly positive in $I_1$ and
\begin{equation}\label{eq:2.10}
 -u_1''+Vu_1=0\quad\hbox{in }I_1,
\end{equation}
then
\begin{equation*}
 \lambda_1(V,I_1)=0,
 \qquad \lambda_1(V,I_2)<0.
\end{equation*}
\end{lemma}

\begin{proof}
For $v\in C_c^\infty(I_1)$, equation \eqref{eq:2.10} and integration by parts yield the ground-state identity
\begin{equation}\label{eq:2.11}
 \int_{I_1}(|v'|^2+Vv^2)\dd x
 =\int_{I_1}u_1^2\left|\left(\frac v{u_1}\right)'\right|^2\dd x
 \ge0.
\end{equation}
By density, the nonnegativity $\int_{I_1}(|v'|^2+Vv^2)\dd x\ge0$ in \eqref{eq:2.11} holds for every $v\in H_0^1(I_1)$. Testing with $u_1$ in \eqref{eq:2.10}, we derive
\begin{equation}\label{eq:2.12}
 \int_{I_1}(|u_1'|^2+Vu_1^2)\dd x=0.
\end{equation}
Equations \eqref{eq:2.9}, \eqref{eq:2.11}, and \eqref{eq:2.12} imply $\lambda_1(V,I_1)=0$.

Let $\widehat u_1$ be the extension of $u_1$ by zero from $I_1$ to $I_2$. Then
\begin{equation*}
 \widehat u_1\in H_0^1(I_2),
 \qquad
 \int_{I_2}(|\widehat u_1'|^2+V\widehat u_1^2)\dd x=0.
\end{equation*}
By \eqref{eq:2.9}, $\lambda_1(V,I_2)\le0$. We aim to prove $\lambda_1(V,I_2)<0$. Suppose on the contrary that $\lambda_1(V,I_2)=0$. Then $\widehat u_1$ would attain the quotient in \eqref{eq:2.9}. The first variation of that quotient at $\widehat u_1$ would therefore lead to
\begin{equation}\label{eq:2.13}
 \int_{I_2}(\widehat u_1'v'+V\widehat u_1v)\dd x=0
 \qquad(v\in H_0^1(I_2)).
\end{equation}
Since $V$ is continuous, we infer from \eqref{eq:2.13} that $\widehat u_1\in C^2(I_2)$ and
\begin{equation*}
 -\widehat u_1''+V\widehat u_1=0\qquad\text{on }I_2.
\end{equation*}
Note that $\widehat u_1$ vanishes on the nonempty open set $I_2\setminus\overline I_1$. At any point of this open set, $\widehat u_1=\widehat u_1'=0$. Uniqueness for the scalar linear Cauchy problem gives $\widehat u_1\equiv0$, contradicting $u_1>0$. Therefore $\lambda_1(V,I_2)<0$.
\end{proof}

The next proposition proves a rigidity result for stable solutions of the exact one-dimensional equation \eqref{eq:2.7}. To be precise, a solution $q$ is \emph{Morse-stable} if $Q_{q,c}\geq 0$ in the sense
\begin{equation*}
 Q_{q,c}(\phi)\ge0
 \qquad\text{for every }\phi\in C_c^\infty(\R,\C).
\end{equation*}

Ortega \cite[Appendix, Proposition~8.1]{BellazziniRuiz} proved the instability of the constant-modulus solutions below the threshold in \eqref{eq:2.14}, and we prove the converse including the equality case and classify all bounded entire Morse-stable solutions.

\begin{proposition}
\label{prop:2.2}
Let $c\in(0,\sqrt2)$ and let $q$ be a bounded solution of $R_c(q)=0$ on $\R$. Then
\begin{equation*}
 Q_{q,c}(\phi)\ge0
 \qquad\text{for every }\phi\in C_c^\infty(\R,\C)
\end{equation*}
if and only if there are constants $\rho>0$ and $\omega,\gamma\in\R$ such that
\begin{equation}\label{eq:2.14}
 \begin{gathered}
 q(x)=\rho\e^{i(\omega x+\gamma)},\\
 \omega^2+c\omega+\rho^2=1,
 \qquad
 \rho^2\ge\frac23\left(1+\frac{c^2}{4}\right).
 \end{gathered}
\end{equation}
\end{proposition}
\begin{proof}
Let $\Phi=\e^{icx/2}q$ and $a=1+c^2/4$, as above. For every solution of \eqref{eq:2.7}, there are two conserved quantities
\begin{equation*}
 m=\Ima(\overline\Phi\Phi'),
 \qquad
 H=\frac12|\Phi'|^2+\frac a2|\Phi|^2-\frac14|\Phi|^4.
\end{equation*}
Using the equation and direct computation, we obtain
\begin{equation*}
 m'=\Ima(\overline\Phi\Phi'')=0,
 \qquad
 H'=\Rea\!\left(
 \overline{\Phi'}[\Phi''+(a-|\Phi|^2)\Phi]\right)=0.
\end{equation*}

We divide the following proof into two cases according to the value of $m$.

\noindent $Case (i)$. Suppose first that $m\ne0$. Then $\Phi$ never vanishes, so we may write
\begin{equation*}
 \Phi=r\e^{i\vartheta},
 \qquad r=|\Phi|>0,
\end{equation*}
where $\vartheta$ is a real-valued phase. The real and imaginary parts of \eqref{eq:2.7} yield
\begin{equation*}
 r^2\vartheta'=m,
 \qquad
 r''=\frac{m^2}{r^3}-ar+r^3.
\end{equation*}
It follows that
\begin{equation*}
 \frac12(r')^2+U_m(r)=H,
 \qquad
 U_m(r)=\frac{m^2}{2r^2}+\frac a2r^2-\frac14r^4.
\end{equation*}

Writing $u=r^2$, we have
\begin{equation*}
 U_m'(r)
 =-\frac{m^2}{r^3}+ar-r^3
 =\frac{r^4(a-r^2)-m^2}{r^3},
 \qquad U_m'(r)=0
 \quad\Longleftrightarrow\quad
 u^2(a-u)=m^2.
\end{equation*}
Since $ \frac{\dd}{\dd u}\bigl[u^2(a-u)\bigr]=u(2a-3u)$, $u^2(a-u)$ increases up to $u=2a/3$, decreases afterwards, and has maximum $4a^3/27$. If $m^2\ge4a^3/27$, then $U_m'\le0$. Since $r''=-U_m'(r)$, boundedness of $\Phi$ on $\R$ forces $r$ to be constant. Assume now $m^2<4a^3/27$, and denote the two critical points satisfying $U_m'=0$ by $r_-<r_+$. The preceding calculation implies
\begin{equation*}
 0<r_-^2<\frac{2a}{3}<r_+^2<a.
\end{equation*}
At either critical point,
\begin{equation*}
 U_m''(r)
 =\frac{3m^2}{r^4}+a-3r^2
 =2(2a-3r^2).
\end{equation*}
Hence $r_-$ is a strict local minimum and $r_+$ is a strict local maximum. Moreover, $U_m$ is strictly decreasing on $(0,r_-)$, strictly increasing on $(r_-,r_+)$, and strictly decreasing on $(r_+,\infty)$. Also,
\begin{equation*}
 \lim_{r\to0^+}U_m(r)=+\infty,
 \qquad
 \lim_{r\to\infty}U_m(r)=-\infty.
\end{equation*}

We now determine the possible solutions directly from
\begin{equation*}
 (r')^2=2\{H-U_m(r)\}.
\end{equation*}
A solution can take only values for which $U_m(r)\le H$, and it can change direction only at a root of $U_m(r)=H$.

Suppose first that
\begin{equation*}
 U_m(r_-)<H<U_m(r_+).
\end{equation*}
Then $U_m(r)=H$ has three simple roots. Denote the first two by $s_1$ and $s_2$. Their positions are
\begin{equation*}
 0<s_1<r_-<s_2<r_+,
\end{equation*}
while the third root lies to the right of $r_+$. The interval $[s_1,s_2]$ is therefore the only bounded connected component of $\{r>0:U_m(r)\le H\}$. Separating variables, we find
\begin{equation*}
 x-x_0
 =\pm\int_{r(x_0)}^{r(x)}
 \frac{\dd s}{\sqrt{2\{H-U_m(s)\}}}.
\end{equation*}
Since $s_1$ and $s_2$ are simple roots, there is $C>1$ such that
\begin{equation*}
 C^{-1}|s-s_j|\le H-U_m(s)\le C|s-s_j|
\end{equation*}
whenever $s\in[s_1,s_2]$ is sufficiently close to $s_j$, where $j=1,2$. Thus the integral defining $x-x_0$ is finite at both endpoints. The function $r(x)$ reaches $s_1$ and $s_2$ at finite values of $x$, and $r'$ changes sign at each endpoint. Repeating this motion between $s_1$ and $s_2$ makes $r$ periodic, with
\begin{equation*}
 s_1\le r(x)\le s_2.
\end{equation*}

Next suppose that
\begin{equation*}
 H=U_m(r_+).
\end{equation*}
Besides the double root $r_+$, the equation $U_m(r)=H$ has one simple root $s_*$ with
\begin{equation*}
 0<s_*<r_-.
\end{equation*}
By Taylor's formula at $r_+$, we obtain
\begin{equation*}
 H-U_m(r)
 =-\frac12U_m''(r_+)(r-r_+)^2
   +O(|r-r_+|^3),
 \qquad U_m''(r_+)<0.
\end{equation*}
Consequently, for some $C>1$,
\begin{equation*}
 C^{-1}(r_+-r)^2\le H-U_m(r)\le C(r_+-r)^2
\end{equation*}
for $r<r_+$ sufficiently close to $r_+$. It follows that
\begin{equation*}
 \int^{r_+}\frac{\dd s}
 {\sqrt{2\{H-U_m(s)\}}}=+\infty.
\end{equation*}
After translating $x$, the corresponding nonconstant bounded solution satisfies
\begin{equation*}
 r(0)=s_*,
 \qquad r'(0)=0,
 \qquad r'(x)<0\quad(x<0),
 \qquad r'(x)>0\quad(x>0),
\end{equation*}
and
\begin{equation*}
 \lim_{x\to-\infty}r(x)
 =\lim_{x\to+\infty}r(x)=r_+.
\end{equation*}

The two critical points also yield the constant solutions
\begin{equation*}
 r\equiv r_-,
 \qquad
 r\equiv r_+.
\end{equation*}
For every remaining value of $H$, or on the branch lying to the right of the largest root of $U_m(r)=H$, the admissible set $\{r>0:U_m(r)\le H\}$ extends to infinity and contains no further turning point. The identity for $(r')^2$ then forces $r$ to be unbounded in at least one direction. Thus every bounded positive solution $r$ is constant, is one of the periodic solutions above, or tends to $r_+$ at both ends.

We now construct a negative test when $r$ is nonconstant and either periodic or tends to $r_+$ at both ends. For real $\alpha,\beta\in C_c^\infty(\R)$, put
\begin{equation*}
 \zeta=\e^{i\vartheta}(\alpha+ir\beta).
\end{equation*}
Then
\begin{equation*}
 \zeta'=\e^{i\vartheta}
 \bigl[(\alpha'-r\vartheta'\beta)
 +i(r'\beta+r\beta'+\vartheta'\alpha)\bigr].
\end{equation*}
Substituting in \eqref{eq:2.8} and integrating by parts, we derive
\begin{align*}
 Q_{q,c}(\e^{-icx/2}\zeta)
 =\int_\R\bigl\{
 &|\alpha'|^2+r^2|\beta'|^2+4r\vartheta'\alpha\beta'
 +(3r^2-a+(\vartheta')^2)\alpha^2\\
 &+2(2r'\vartheta'+r\vartheta'')\alpha\beta
 +(r^2(\vartheta')^2-rr''-ar^2+r^4)\beta^2
 \bigr\}\dd x.
\end{align*}
The two equations for $r$ and $\vartheta$ make the coefficients of $\alpha\beta$ and $\beta^2$ equal to zero. Completing the one remaining square,
\begin{equation*}
 r^2|\beta'|^2+4r\vartheta'\alpha\beta'
 =r^2\left(\beta'+\frac{2\vartheta'}r\alpha\right)^2
 -4(\vartheta')^2\alpha^2,
\end{equation*}
therefore yields the exact formula
\begin{equation*}
 Q_{q,c}\!\left(
 \e^{-icx/2+i\vartheta}(\alpha+ir\beta)\right)
 =\int_\R\left\{
 r^2\left(\beta'+\frac{2\vartheta'}r\alpha\right)^2
 +|\alpha'|^2+\bigl(3r^2-a-3(\vartheta')^2\bigr)\alpha^2
 \right\}\dd x.
\end{equation*}

Set
\begin{equation*}
 L=-\frac{\dd^2}{\dd x^2}+3r^2-a-3(\vartheta')^2.
\end{equation*}
Differentiating the equation for $r$ and using $\vartheta'=m/r^2$, we obtain $Lr'=0$. If $r$ is periodic, choose three consecutive zeros $t_0<t_1<t_2$ of $r'$. One of $r'$ and $-r'$ is positive on $(t_0,t_1)$. Lemma~\ref{lem:2.1}, applied first on $(t_0,t_1)$ and then on $(t_0,t_2)$, implies
\begin{equation*}
 \inf_{0\ne\alpha\in H_0^1(t_0,t_2)}
 \frac{\displaystyle\int_{t_0}^{t_2}
 \bigl(|\alpha'|^2+(3r^2-a-3(\vartheta')^2)\alpha^2\bigr)\dd x}
 {\displaystyle\int_{t_0}^{t_2}\alpha^2\dd x}<0.
\end{equation*}

For the nonconstant solution tending to $r_+$ at both ends, write $r_*=r_+$. Here $r_*^2>2a/3$. Since $U_m''(r_*)<0$, Taylor's formula and the identity $\frac12(r')^2+U_m(r)=H$ imply that, for some $C,\mu>0$,
\begin{equation*}
 |r(x)-r_*|+|r'(x)|+|r''(x)|\le C\e^{-\mu|x|}
 \qquad (|x|\text{ sufficiently large}).
\end{equation*}
Hence $r'\in H^1(\R)$, $r'$ has exactly one zero, and
\begin{equation*}
 3r^2-a-3(\vartheta')^2
 \longrightarrow 2(3r_*^2-2a)>0.
\end{equation*}
We regard $L$ as the self-adjoint operator on $L^2(\R)$ with domain $H^2(\R)$. Its potential converges exponentially to $2(3r_*^2-2a)>0$. The essential-spectrum theorem for one-dimensional Schrödinger operators \cite[Sections~6.4 and~9.7]{TeschlSchrodinger} implies
\begin{equation*}
 \sigma_{\mathrm{ess}}(L)
 =[2(3r_*^2-2a),\infty).
\end{equation*}
From the equation for $r$ and the preceding exponential estimate, we also infer that $r'\in H^2(\R)$. Thus $Lr'=0$ and zero is a discrete eigenvalue below the essential spectrum. Since $r'$ has exactly one zero, the one-dimensional Sturm nodal theorem \cite[Theorem~9.40]{TeschlSchrodinger} shows that zero is the second discrete eigenvalue of $L$. Hence $L$ has an eigenvalue $\lambda_-<0$. Approximating its eigenfunction by a smooth compactly supported function yields the same conclusion as in the periodic case. Thus, in either case, there are $\alpha\in C_c^\infty(\R,\R)$ and $\nu>0$ such that
\begin{equation*}
 \int_\R\bigl\{|\alpha'|^2+
 (3r^2-a-3(\vartheta')^2)\alpha^2\bigr\}\dd x=-\nu.
\end{equation*}

It remains to choose $\beta$ so that $\int_\R r^2(\beta'+2\vartheta'\alpha/r)^2\dd x<\nu$. Choose $b$ to the right of the support of $\alpha$, let
\begin{equation*}
 A=\int_\R\frac{2\vartheta'}r\alpha\dd x,
\end{equation*}
and take $\eta\in C^\infty(\R,[0,1])$ with $\eta=1$ on $(-\infty,0]$ and $\eta=0$ on $[1,\infty)$. Define
\begin{equation*}
 \beta_R(x)=
 \begin{cases}
 -\displaystyle\int_{-\infty}^{x}
  \dfrac{2\vartheta'(s)}{r(s)}\alpha(s)\dd s,&x\le b,\\[6pt]
 -A\eta\!\left(\dfrac{x-b}{R}\right),&x>b.
 \end{cases}
\end{equation*}
Then $\beta_R\in C_c^\infty(\R)$ and
\begin{equation*}
 \int_\R r^2\left(\beta_R'
 +\frac{2\vartheta'}r\alpha\right)^2\dd x
 \le \frac{C}{R}.
\end{equation*}
For large $R$, the constructed test therefore satisfies
\begin{equation*}
 Q_{q,c}\!\left(\e^{-icx/2+i\vartheta}(\alpha+ir\beta_R)\right)
 \le-\nu+C/R<0.
\end{equation*}

\bigskip

\noindent $Case (ii)$. We now turn to $m=0$. On every interval on which $\Phi\ne0$, its phase is constant because
\begin{equation*}
 \frac{\dd}{\dd x}\arg\Phi
 =\frac{\Ima(\overline\Phi\Phi')}{|\Phi|^2}=0.
\end{equation*}
If $\Phi$ never vanishes, then $\arg\Phi$ is constant on $\R$, and a constant rotation makes $\Phi$ real on $\R$. Suppose now that $\Phi(x_0)=0$ for some point $x_0$. If $\Phi'(x_0)=0$, we conclude from uniqueness for \eqref{eq:2.7} that $\Phi\equiv0$. Otherwise, multiply $\Phi$ by a constant of modulus one so that $\Phi'(x_0)$ is real. The rotated solution and its complex conjugate then have the same Cauchy data at $x_0$. Uniqueness for \eqref{eq:2.7} shows that the rotated solution is real on $\R$. It follows that, after a constant rotation, we may write
\begin{equation*}
 \Phi=f\in\R,
 \qquad
 f''+af-f^3=0,
 \qquad
 \frac12(f')^2+\frac a2f^2-\frac14f^4=H.
\end{equation*}
The function $ s\longmapsto \frac a2s^2-\frac14s^4$ has value $0$ at $s=0$, attains its maximum $a^2/4$ at $s=\pm\sqrt a$, and tends to $-\infty$ as $|s|\to\infty$. Together with
\begin{equation*}
 (f')^2=2\left(H-\frac a2f^2+\frac14f^4\right),
\end{equation*}
this determines all bounded solutions. They are $f\equiv0$, corresponding to $H=0$, the periodic solutions corresponding to $0<H<a^2/4$, and the constant solutions $f\equiv\pm\sqrt a$ and nonconstant solutions corresponding to $H=a^2/4$. In particular, if $H=a^2/4$ and $f$ is nonconstant, then
\begin{equation*}
 (f')^2=\frac12(a-f^2)^2,
\end{equation*}
and direct integration gives
\begin{equation*}
 f(x)=\pm\sqrt a\tanh\!\left(
 \sqrt{\frac a2}(x-t)\right),
 \qquad t\in\R.
\end{equation*}
It remains to exclude $H<0$ and $H>a^2/4$. If $H<0$, the equation
\begin{equation*}
 \frac a2s^2-\frac14s^4=H
\end{equation*}
has exactly two real roots $\pm b$, where
\begin{equation*}
 b^2=a+\sqrt{a^2-4H}>2a.
\end{equation*}
The energy identity implies that $|f|\ge b$. If $f\ge b$, then
\begin{equation*}
 f''=f(f^2-a)\ge b(b^2-a)>0,
\end{equation*}
which is incompatible with boundedness on $\R$. The case $f\le-b$ is identical.

If $H>a^2/4$, then
\begin{equation*}
 |f'|^2
 =2\left(H-\frac a2f^2+\frac14f^4\right)
 \ge2\left(H-\frac{a^2}{4}\right)>0.
\end{equation*}
Thus $f'$ has constant sign and is bounded away from zero, so $f$ cannot be bounded on $\R$.

We now prove that $f\equiv0$ and every nonconstant bounded solution listed above admit a compactly supported variation on which $Q_{q,c}$ is negative. Since $f$ is real, the real and imaginary parts of a variation decouple. More precisely, for $\alpha,\beta\in C_c^\infty(\R,\R)$, we deduce from formula \eqref{eq:2.8} that
\begin{equation*}
 Q_{q,c}(\e^{-icx/2}(\alpha+i\beta))
 =\int_\R\bigl\{|\alpha'|^2+(3f^2-a)\alpha^2\bigr\}\dd x
 +\int_\R\bigl\{|\beta'|^2+(f^2-a)\beta^2\bigr\}\dd x.
\end{equation*}
For a nonconstant periodic $f$, let $t_0<t_1<t_2$ be three consecutive zeros of $f$. One of $f$ and $-f$ is positive on $(t_0,t_1)$, and
\begin{equation*}
 -f''+(f^2-a)f=0,
\end{equation*}
Lemma~\ref{lem:2.1}, with $(t_0,t_1)\subset(t_0,t_2)$, provides a compactly supported $\beta$ for which the second integral is negative. For the displayed hyperbolic-tangent solution,
\begin{equation*}
 f^2-a=-a\,\operatorname{sech}^2\!\left(
 \sqrt{\frac a2}(x-t)\right).
\end{equation*}
Taking $\beta$ equal to $1$ on $[-R,R]$ and cutting it off on $[-2R,2R]$ makes the second integral converge to
\begin{equation*}
 -a\int_\R\operatorname{sech}^2\!\left(
 \sqrt{\frac a2}(x-t)\right)\dd x<0.
\end{equation*}
The same argument by cutoff yields a negative value for $f\equiv0$. Therefore the only cases not yet ruled out are the nonzero constant-modulus solutions.

Every such solution can be written
\begin{equation*}
 q(x)=\rho\e^{i(\omega x+\gamma)},
 \qquad \rho>0,\quad \omega,\gamma\in\R.
\end{equation*}
Substituting in $R_c(q)=0$, we obtain
\begin{equation*}
 \omega^2+c\omega+\rho^2=1,
 \qquad
 \rho^2+\left(\omega+\frac c2\right)^2=a.
\end{equation*}
Substituting $r=\rho$ and $\vartheta'=\omega+c/2$ in the completed-square formula for $Q_{q,c}$, we obtain
\begin{align*}
 &Q_{q,c}\!\left(
 \e^{i(\omega x+\gamma)}(\alpha+i\rho\beta)\right)\\
 &\quad=\int_\R\left\{
 \rho^2\left(\beta'
 +\frac{2(\omega+c/2)}{\rho}\alpha\right)^2
 +|\alpha'|^2
 +2\left[\rho^2-2\left(\omega+\frac c2\right)^2\right]\alpha^2
 \right\}\dd x.
\end{align*}
If $\rho^2\ge2(\omega+c/2)^2$, every term on the right is nonnegative. If the reverse strict inequality holds, choose $\alpha_R\in C_c^\infty(\R)$ equal to $1$ on $[0,R]$, with transition regions of fixed length. Then
\begin{equation*}
 \int_\R\left\{|\alpha_R'|^2
 +2\left[\rho^2-2\left(\omega+\frac c2\right)^2\right]
 \alpha_R^2\right\}\dd x=-C_0R+O(1)
\end{equation*}
for some $C_0>0$. Apply the preceding construction of $\beta_R$, now returning it to zero over an interval of length $R^3$. Since $\int\alpha_R=O(R)$, the integral of $\rho^2(\beta_R'+2(\omega+c/2)\alpha_R/\rho)^2$ is only $O(R^{-1})$. Hence $Q_{q,c}<0$ for all sufficiently large $R$.

We have proved that the constant-modulus solution is stable exactly when
\begin{equation*}
 \rho^2\ge2\left(\omega+\frac c2\right)^2.
\end{equation*}
Together with $\rho^2+(\omega+c/2)^2=a$, this is equivalent to $\rho^2\ge2a/3$. Since $a=1+c^2/4$, these are precisely the conditions in \eqref{eq:2.14}.
\end{proof}

The explicit formula for stable solutions gives the following estimate.
\begin{lemma}\label{lem:2.3}
For every compact interval $J\subset(0,\sqrt2)$ there is $\delta_J>0$ such that, if $c\in J$ and $q$ satisfies \eqref{eq:2.14}, then
\begin{equation}\label{eq:2.15}
 D_q\ge\delta_JF_q.
\end{equation}
\end{lemma}

\begin{proof}
It follows from the algebraic relation in \eqref{eq:2.14} that
\begin{equation*}
 1-\rho^2=\omega(\omega+c),
 \qquad
 \rho^2=a-\left(\omega+\frac c2\right)^2.
\end{equation*}
If $\omega\ne0$, substituting in \eqref{eq:2.4}, we obtain
\begin{equation*}
 \frac{D_q}{F_q}
 =\frac{\rho^2-\frac12(\omega+c)^2}
 {\rho^2+\frac12(\omega+c)^2}
 =\frac{a-(\omega+c/2)^2-\frac12(\omega+c)^2}
 {a-(\omega+c/2)^2+\frac12(\omega+c)^2}.
\end{equation*}
The last condition in \eqref{eq:2.14} is equivalent to
\begin{equation*}
 \left|\omega+\frac c2\right|\le\sqrt{\frac a3}.
\end{equation*}
On this interval, the numerator in the last quotient is a concave quadratic function of $\omega+c/2$. Its minimum is
\begin{equation*}
 \frac12\left(1-c\sqrt{\frac a3}\right)>0
 \qquad(c<\sqrt2).
\end{equation*}
The rational expression on the right extends continuously to $\omega=0$, where $\rho=1$ and its value is $(2-c^2)/(2+c^2)>0$. The displayed bound on $|\omega+c/2|$ makes the admissible pairs $(c,\omega)$ with $c\in J$ a compact set. The quotient therefore has a positive lower bound $\delta_J$ there. This proves \eqref{eq:2.15}. When $(\rho,\omega)=(1,0)$, both sides of \eqref{eq:2.15} are zero.
\end{proof}

Proposition~\ref{prop:2.2} and Lemma~\ref{lem:2.3} concern exact Morse-stable solutions on the whole line. We next prove an estimate near a constant of modulus one. A constant phase rotation reduces that constant to $1$. Write a perturbation as $q=1+A+iB$, where $A,B:\R\to\R$. The terms that are linear in $(A,B)$ in $R_c(q)=0$ yield
\begin{equation}\label{eq:2.16}
 A''-cB'-2A=0,
 \qquad B''+cA'=0.
\end{equation}
Define the positive subsonic exponent
\begin{equation*}
 \lambda_c:=\sqrt{2-c^2}>0.
\end{equation*}
For an exponential mode $(A,B)=(\xi,\eta)\e^{\lambda x}$, system \eqref{eq:2.16} becomes the following equation for the column vector with entries $\xi,\eta$:
\begin{equation}\label{eq:2.17}
 \begin{pmatrix}
  \lambda^2-2&-c\lambda\\
  c\lambda&\lambda^2
 \end{pmatrix}
 \binom{\xi}{\eta}=0.
\end{equation}
Writing $\det$ for the determinant, the characteristic polynomial is
\begin{equation*}
 \det\begin{pmatrix}
  \lambda^2-2&-c\lambda\\
  c\lambda&\lambda^2
 \end{pmatrix}
 =\lambda^2(\lambda^2-\lambda_c^2).
\end{equation*}
The double root $\lambda=0$ corresponds to the constant phase mode $(A,B)=(0,\gamma)$ and the affine mode $(A,B)=(-c\beta/2,\beta x)$. At $\lambda=\lambda_c$ and $\lambda=-\lambda_c$, \eqref{eq:2.17} implies $\eta=-c\xi/\lambda_c$ and $\eta=c\xi/\lambda_c$, respectively. Thus every real solution of \eqref{eq:2.16} has the form
\begin{align}
 A(x)&=-\frac c2\beta+a_+\e^{\lambda_cx}
 +a_-\e^{-\lambda_cx},\label{eq:2.18}\\
 B(x)&=\gamma+\beta x
 -\frac c{\lambda_c}a_+\e^{\lambda_cx}
 +\frac c{\lambda_c}a_-\e^{-\lambda_cx}.
 \label{eq:2.19}
\end{align}
Here $\beta,\gamma,a_+,a_-\in\R$. Define
\begin{equation}\label{eq:2.20}
 F_0=|A'|^2+|B'|^2+2A^2,
 \qquad D_0=|A'|^2+|B'|^2-2A^2.
\end{equation}
For the constant phase mode $(A,B)=(0,\gamma)$, both densities in \eqref{eq:2.20} vanish. For the affine mode $(A,B)=(-c\beta/2,\beta x)$,
\begin{equation*}
 D_0=\left(1-\frac{c^2}{2}\right)\beta^2,
 \qquad
 F_0=\left(1+\frac{c^2}{2}\right)\beta^2.
\end{equation*}
For the mode proportional to $\e^{\lambda_cx}$ we have $A'=\lambda_cA$ and $B'=-cA$. For the mode proportional to $\e^{-\lambda_cx}$ we have $A'=-\lambda_cA$ and $B'=-cA$. Thus we conclude from \eqref{eq:2.20} that
\begin{equation*}
 D_0=(\lambda_c^2+c^2-2)A^2=0.
\end{equation*}
The affine mode of \eqref{eq:2.16} satisfies a uniform positive estimate relating $D_0$ to $F_0$. Each exponential mode forces the $F_0$-energy to grow toward one of the adjacent intervals. The following lemma establishes a uniform alternative.

\begin{lemma}\label{lem:2.4}
Let $J\subset(0,\sqrt2)$ be a compact interval and $K>1$. There are $L=L(J,K)$ and $\delta_0=\delta_0(J)>0$ with the following property. Let $\mathbb Z$ denote the set of integers. For $j\in\mathbb Z$, put $I_j=[jL,(j+1)L]$. Every solution of \eqref{eq:2.16} with $c\in J$ satisfies, on each $I_j$,
\begin{equation}\label{eq:2.21}
 \int_{I_j}D_0\ge2\delta_0
 \int_{I_j}F_0
\end{equation}
or
\begin{equation}\label{eq:2.22}
 \max\left\{
 \int_{I_{j-1}}F_0,
 \int_{I_{j+1}}F_0\right\}
 \ge2K\int_{I_j}F_0.
\end{equation}
\end{lemma}

\begin{proof}
Translation reduces the proof to the central cell $I_0=[0,L]$. Put $t=x-L/2$. Equations \eqref{eq:2.18}--\eqref{eq:2.19} can then be written as
\begin{align*}
 A&=-\frac c2\beta+p_+\e^{\lambda_ct}
 +p_-\e^{-\lambda_ct},\\
 B&=\gamma_0+\beta t-\frac c{\lambda_c}p_+\e^{\lambda_ct}
 +\frac c{\lambda_c}p_-\e^{-\lambda_ct},
\end{align*}
where $\beta,\gamma_0,p_+,p_-\in\R$. Substituting in \eqref{eq:2.20} and using $\lambda_c^2=2-c^2$, we derive
\begin{align}
 F_0
 &=\lambda_c^2(p_+\e^{\lambda_ct}-p_-\e^{-\lambda_ct})^2
 +\left(1+\frac{c^2}{2}\right)\beta^2\notag\\
 &\quad-4c\beta(p_+\e^{\lambda_ct}+p_-\e^{-\lambda_ct})
 +(2+c^2)(p_+\e^{\lambda_ct}+p_-\e^{-\lambda_ct})^2,
 \notag\\
 D_0
 &=\lambda_c^2\left(\frac{\beta^2}{2}-4p_+p_-\right).
 \label{eq:2.23}
\end{align}
Thus $D_0$ is constant in $x$. The determinant of the quadratic part in $\beta$ and $p_+\e^{\lambda_ct}+p_-\e^{-\lambda_ct}$ is
\begin{equation*}
 \left(1+\frac{c^2}{2}\right)(2+c^2)-4c^2
 =\frac12(2-c^2)^2.
\end{equation*}
Since $J$ is compact in $(0,\sqrt2)$, there is $C_J\ge1$ such that, pointwise,
\begin{equation}\label{eq:2.24}
 C_J^{-1}\{\beta^2+p_+^2\e^{2\lambda_ct}
 +p_-^2\e^{-2\lambda_ct}\}
 \le F_0\le C_J\{\beta^2+p_+^2\e^{2\lambda_ct}
 +p_-^2\e^{-2\lambda_ct}\}.
\end{equation}

Since $\int_{-L/2}^{L/2}\e^{2\lambda_ct}\dd t =\sinh(\lambda_cL)/\lambda_c$, we conclude from integration of \eqref{eq:2.24} on the central cell that
\begin{equation}\label{eq:2.25}
 C_J^{-1}\left\{L\beta^2+
 \frac{\sinh(\lambda_cL)}{\lambda_c}(p_+^2+p_-^2)\right\}
 \le \int_{I_0}F_0
 \le C_J\left\{L\beta^2+
 \frac{\sinh(\lambda_cL)}{\lambda_c}(p_+^2+p_-^2)\right\}.
\end{equation}
Applying \eqref{eq:2.24} on $I_{-1}$ and $I_1$ and integrating the two exponential squares, we infer
\begin{align}
 &\int_{I_{-1}}F_0+\int_{I_1}F_0\ge C_J^{-1}\left\{2L\beta^2
 +2\cosh(2\lambda_cL)\frac{\sinh(\lambda_cL)}{\lambda_c}
 (p_+^2+p_-^2)\right\}.
 \label{eq:2.26}
\end{align}

Assume that \eqref{eq:2.22} fails, and put $\lambda_*=\min_{c\in J}\lambda_c>0$. The failure of \eqref{eq:2.22} implies
\begin{equation}\label{eq:2.27}
 \int_{I_{-1}}F_0+\int_{I_1}F_0
 <4K\int_{I_0}F_0.
\end{equation}
Substitution of \eqref{eq:2.25} and \eqref{eq:2.26} in \eqref{eq:2.27} yields
\begin{equation}\label{eq:2.28}
 \begin{aligned}
 &C_J^{-1}\left\{2L\beta^2+
 2\cosh(2\lambda_cL)\frac{\sinh(\lambda_cL)}{\lambda_c}
 (p_+^2+p_-^2)\right\}\\
 &\qquad<4KC_J\left\{L\beta^2+
 \frac{\sinh(\lambda_cL)}{\lambda_c}(p_+^2+p_-^2)\right\}.
 \end{aligned}
\end{equation}
Choose $L$ so large that
\begin{equation*}
 C_J^{-1}\cosh(2\lambda_*L)>4KC_J.
\end{equation*}
Moving the exponential term in \eqref{eq:2.28} to the left and using $\lambda_c\ge\lambda_*$, we obtain
\begin{equation}\label{eq:2.29}
 \frac{\sinh(\lambda_cL)}{\lambda_c}(p_+^2+p_-^2)
 \le C_{J,K}\e^{-2\lambda_*L}L\beta^2.
\end{equation}
Increase $L$, if necessary, so that the coefficient of $L\beta^2$ on the right of \eqref{eq:2.29} is at most $1/8$. Since $\sinh(\lambda_cL)/\lambda_c\ge L$ and $4|p_+p_-|\le2(p_+^2+p_-^2)$, it follows from equations \eqref{eq:2.29} and \eqref{eq:2.23} that
\begin{align}
 \int_{I_0}D_0
 &=L\lambda_c^2\left(\frac{\beta^2}{2}-4p_+p_-\right)
 \ge\frac{\lambda_*^2}{4}L\beta^2,
 \label{eq:2.30}\\
 \int_{I_0}F_0
 &\le C_J\left\{L\beta^2+
 \frac{\sinh(\lambda_cL)}{\lambda_c}(p_+^2+p_-^2)\right\}
 \le2C_JL\beta^2.
 \label{eq:2.31}
\end{align}
Equations \eqref{eq:2.30} and \eqref{eq:2.31} prove \eqref{eq:2.21} with
\begin{equation*}
 \delta_0=\frac{\lambda_*^2}{16C_J},
\end{equation*}
which is independent of $K$. If $\beta=0$, \eqref{eq:2.29} forces $p_+=p_-=0$, and both sides of \eqref{eq:2.21} vanish. This completes the proof.
\end{proof}

\subsection{The inhomogeneous one-dimensional equation}

Fix a compact interval $J=[c_-,c_+]\subset(0,\sqrt2)$. Fix also a number $M_J$, depending only on $J$, and consider functions $q$ satisfying
\begin{equation}\label{eq:2.32}
 \sum_{k=0}^3\|q^{(k)}\|_{L^\infty(\R)}\le M_J.
\end{equation}
Here $q^{(k)}$ denotes the $k$-th derivative in $x$. Section~3 verifies \eqref{eq:2.32} for $q=q_y$ using the universal elliptic estimates for \eqref{eq:1.2}. Every constant in this subsection is uniform for $c\in J$. For a cell length $L>1$ define
\begin{equation*}
 I_j=[jL,(j+1)L].
\end{equation*}
We use the three-cell window $[(j-1)L,(j+2)L]$ and, for an integer $M\ge2$, the enlarged interval
\begin{equation*}
 J_j=((j-M)L,(j+M+1)L).
\end{equation*}
Once $L$ and $M$ have been fixed, set
\begin{equation}\label{eq:2.33}
 f_j=\int_{I_j}F_q,
 \qquad d_j=\int_{I_j}D_q,
 \qquad \mathcal R_j=\int_{J_j}|R_c(q)|^2.
\end{equation}
Thus $f_j\ge0$, $d_j\in\R$, and $\mathcal R_j\ge0$.

Two local estimates are needed before the cell argument.

\begin{lemma}
\label{lem:2.5}
Let $I$ be a fixed bounded interval and $B>0$. Uniformly among the functions $q\in C^2(I,\C)$ with $\|q\|_{C^2(I)}\le B$,
\begin{equation}\label{eq:2.34}
 \inf_{\gamma\in\R}\|q-\e^{i\gamma}\|_{C^1(I)}\longrightarrow0
 \quad\hbox{as}\quad
 \int_I\left(|q'|^2+(1-|q|^2)^2\right)\longrightarrow0.
\end{equation}
\end{lemma}

\begin{proof}
All norms in this proof are taken over $I$. We use the elementary one-dimensional interpolation estimate
\begin{equation}\label{eq:2.35}
 \|g\|_\infty
 \le C_{I}\left(\|g\|_2+
 \|g\|_2^{2/3}\|g'\|_\infty^{1/3}\right)
 \qquad(g,g'\in L^\infty(I)).
\end{equation}
To prove \eqref{eq:2.35}, take a point of $\overline I$ at which the continuous representative of $|g|$ attains its maximum. If $g$ is not constant, then on a one-sided interval adjacent to that point, of length
\begin{equation*}
 \min\left\{\frac{|I|}{2},
 \frac{\|g\|_\infty}{2\|g'\|_\infty}\right\},
\end{equation*}
we have $|g|\ge\|g\|_\infty/2$. Integration yields \eqref{eq:2.35}. If $g$ is constant, its first term already proves the estimate.

Suppose that $\int_I(|q'|^2+(1-|q|^2)^2)\dd x\le s$. Then
\begin{equation*}
 \|q'\|_2,\ \|1-|q|^2\|_2\le s^{1/2},
 \qquad \|q''\|_\infty\le B,
 \qquad \|(1-|q|^2)'\|_\infty
 =\|2\Rea(q'\overline q)\|_\infty\le2B^2.
\end{equation*}
Applying \eqref{eq:2.35} first to $g=q'$ and then to $g=1-|q|^2$, we obtain
\begin{equation*}
 \|q'\|_\infty+\|1-|q|^2\|_\infty
 \le C_{I,B}(s^{1/2}+s^{1/3}).
\end{equation*}
The constant $C_{I,B}$ depends only on $I$ and $B$, not on $q$ or $s$. Thus, for the entire class $\|q\|_{C^2(I)}\le B$, the same estimate proves
\begin{equation}\label{eq:2.36}
 \|q'\|_\infty+\|1-|q|^2\|_\infty\longrightarrow0
 \qquad(s\downarrow0).
\end{equation}
Choose $x_0\in I$. For $s$ sufficiently small, \eqref{eq:2.36} ensures $|q|^2\ge1/2$ on $I$, so $q(x_0)\ne0$. Set $\e^{i\gamma}=q(x_0)/|q(x_0)|$. The triangle inequality and $|q(x)-q(x_0)|\le |I|\,\|q'\|_\infty$ yield
\begin{equation}\label{eq:2.37}
 |q(x)-\e^{i\gamma}|
 \le |I|\,\|q'\|_\infty
 +\bigl||q(x_0)|-1\bigr|.
\end{equation}
Since $\bigl||q(x_0)|-1\bigr|\le|1-|q(x_0)|^2|$, equations \eqref{eq:2.36} and \eqref{eq:2.37} imply $\|q-\e^{i\gamma}\|_\infty+\|q'\|_\infty\to0$, which is \eqref{eq:2.34}.
\end{proof}

\begin{lemma}
\label{lem:2.6}
Fix $K>1$ and apply Lemma~\ref{lem:2.4} with the growth factor $4K$, and denote its cell length by $L$ and the constant in \eqref{eq:2.21} by $\delta_0$. There are $\sigma>0$ and $\varepsilon>0$, depending only on $J$ and $K$, such that the following holds. Let $c\in J$ and $j\in\mathbb Z$, and let $q$ satisfy \eqref{eq:2.32}. Suppose that, after multiplication by a constant phase,
\begin{equation}\label{eq:2.38}
 \|q-1\|_{L^\infty([(j-1)L,(j+2)L])}\le\sigma,
 \qquad
 \int_{(j-1)L}^{(j+2)L}|R_c(q)|^2\le\varepsilon f_j.
\end{equation}
Then either
\begin{equation}\label{eq:2.39}
 d_j\ge\frac{\delta_0}{2}f_j
\end{equation}
or
\begin{equation}\label{eq:2.40}
 \max\{f_{j-1},f_{j+1}\}\ge Kf_j.
\end{equation}
\end{lemma}

\begin{proof}
By translation, it suffices to consider $j=0$. If $f_0=0$, then $q$ is constant with modulus one on $I_0$, and \eqref{eq:2.39} holds.

Suppose that the conclusion is false. Then there are $c_n\in J$ and functions $q_n$ such that, after a constant phase rotation,
\begin{equation*}
 \|q_n-1\|_{L^\infty([-L,2L])}\le\frac1n,
 \qquad
 \int_{-L}^{2L}|R_{c_n}(q_n)|^2\dd x
 \le\frac1n f_{0,n},
\end{equation*}
while
\begin{equation*}
 d_{0,n}<\frac{\delta_0}{2}f_{0,n},
 \qquad
 \max\{f_{-1,n},f_{1,n}\}<Kf_{0,n}.
\end{equation*}
Here $f_{k,n}$ and $d_{k,n}$ are defined by \eqref{eq:2.33} with $q=q_n$. Necessarily $f_{0,n}>0$.

Fix $x_0\in I_0$. A further phase rotation, which preserves $f_{k,n}$, $d_{k,n}$, and $|R_{c_n}(q_n)|$, allows us to assume
\begin{equation*}
 q_n(x_0)>0,
 \qquad
 \|q_n-1\|_{L^\infty([-L,2L])}\longrightarrow0.
\end{equation*}
Write
\begin{equation*}
 q_n=1+z_n,
 \qquad
 \Im z_n(x_0)=0,
 \qquad
 s_n=f_{0,n}^{1/2},
 \qquad
 u_n=\frac{z_n}{s_n}.
\end{equation*}
The identity
\begin{equation*}
 \frac12(1-|1+z|^2)^2
 =2(\Re z)^2+2(\Re z)|z|^2+\frac12|z|^4
\end{equation*}
gives
\begin{equation*}
 \begin{aligned}
 \bigl|F_{q_n}-|z_n'|^2-2(\Re z_n)^2\bigr|
 &\le C\|z_n\|_\infty|z_n|^2,\\
 \bigl|D_{q_n}-|z_n'|^2+2(\Re z_n)^2\bigr|
 &\le C\|z_n\|_\infty|z_n|^2.
 \end{aligned}
\end{equation*}
All norms in the next estimates are taken over $[-L,2L]$. Since $\Im z_n(x_0)=0$, Poincar\'e's inequality gives $\|\Im z_n\|_2\le C_{J,K}\|z_n'\|_2$. The expansion of $F_{q_n}$ then implies
\begin{equation*}
 \|z_n\|_{H^1}^2
 \le C_{J,K}\int_{-L}^{2L}F_{q_n}\dd x
 +C_{J,K}\|z_n\|_\infty\|z_n\|_2^2.
\end{equation*}
Absorb the last term using $\|z_n\|_\infty\to0$. The failure of \eqref{eq:2.40} means that $\int_{-L}^{2L}F_{q_n}\dd x=f_{-1,n}+f_{0,n}+f_{1,n}<(1+2K)s_n^2$. Dividing by $s_n^2$ therefore yields
\begin{equation*}
 \|u_n\|_{H^1([-L,2L])}\le C_{J,K}.
\end{equation*}

The equation for $u_n$ is
\begin{equation*}
 u_n''+ic_nu_n'-2\Re u_n
 =\frac{R_{c_n}(q_n)}{s_n}
  +2(\Re z_n)u_n+|z_n||u_n|(1+z_n).
\end{equation*}
Moreover,
\begin{equation*}
 \left\|2(\Re z_n)u_n+|z_n||u_n|(1+z_n)\right\|_2
 \le(3+\|z_n\|_\infty)\|z_n\|_\infty\|u_n\|_2\longrightarrow0.
\end{equation*}
Since $\|R_{c_n}(q_n)/s_n\|_2\le n^{-1/2}$, the right-hand side tends to zero in $L^2([-L,2L])$. Hence $(u_n)$ is bounded in $H^2([-L,2L])$. After passing to a subsequence,
\begin{equation*}
 c_n\longrightarrow c\in J,
 \qquad
 u_n\longrightarrow u
 \quad\text{strongly in }H^1([-L,2L]),
\end{equation*}
and
\begin{equation*}
 u''+icu'-2\Re u=0.
\end{equation*}

Dividing the preceding energy expansions by $s_n^2=f_{0,n}$ and passing to the limit, we conclude that
\begin{equation*}
 \int_{I_0}\bigl(|u'|^2+2(\Re u)^2\bigr)\dd x=1,
\end{equation*}
together with
\begin{equation*}
 \int_{I_0}\bigl(|u'|^2-2(\Re u)^2\bigr)\dd x
 \le\frac{\delta_0}{2}
\end{equation*}
and
\begin{equation*}
 \max_{k=-1,1}
 \int_{I_k}\bigl(|u'|^2+2(\Re u)^2\bigr)\dd x
 \le K.
\end{equation*}
Apply Lemma~\ref{lem:2.4} to $(A,B)=(\Re u,\Im u)$ with growth factor $4K$. Since $\int_{I_0}F_0=1$, it requires either
\begin{equation*}
 \int_{I_0}\bigl(|u'|^2-2(\Re u)^2\bigr)\dd x
 \ge2\delta_0,
 \qquad\text{or}\qquad
 \max_{k=-1,1}\int_{I_k}\bigl(|u'|^2+2(\Re u)^2\bigr)\dd x
 \ge8K.
\end{equation*}
The first contradicts the bound $\delta_0/2$ for the central $D_0$ integral, and the second contradicts the bound $K$ for the neighboring $F_0$ integrals.
\end{proof}

\begin{lemma}
\label{lem:2.7}
Fix the length $L$ selected in Lemma~\ref{lem:2.6} and a number $\sigma>0$. Let $0<\delta<\delta_J/16$, where $\delta_J$ is the constant in Lemma~\ref{lem:2.3}. There exist an integer $M\ge2$ and constants $\kappa>0$, $C_0<\infty$ with the following property. Let $c\in J$ and $j\in\mathbb Z$, and let $q_0$ be an exact solution of $R_c(q_0)=0$ whose state $(q_0,q_0')$ is bounded by $2M_J$ on $J_j$ and by $M_J$ at the center $(j+\tfrac12)L$, where the state norm is the Euclidean norm on $\C^2$. Then at least one of the following holds:
\begin{enumerate}
\item after a constant rotation,
\begin{equation}\label{eq:2.53}
 \|q_0-1\|_{C^1([(j-1)L,(j+2)L])}<\frac\sigma2.
\end{equation}
\item
\begin{equation*}
 \int_{I_j}D_{q_0}\ge4\delta\int_{I_j}F_{q_0}.
\end{equation*}
\item there is $\phi\in C_c^\infty(J_j,\C)$ such that
\begin{equation*}
 \|\phi\|_2=1,\qquad \|\phi'\|_2\le C_0,\qquad
 Q_{q_0,c}(\phi)\le-4\kappa.
\end{equation*}
\end{enumerate}
\end{lemma}

\begin{proof}
Translate the center of the cell to the origin. Suppose the conclusion were false. For each integer $n\ge2$, use $M=n$, $\kappa=1/n$, and $C_0=n$ to choose $c_n\in J$ and an exact solution $q_n$ on
\begin{equation*}
 (-(n+\tfrac12)L,(n+\tfrac12)L)
\end{equation*}
whose state is bounded by $2M_J$ there and by $M_J$ at the origin, but for which all three alternatives fail. Explicitly,
\begin{align}
 \inf_{\gamma\in\R}
 \|\e^{-i\gamma}q_n-1\|_{C^1([-3L/2,3L/2])}
 &\ge\frac{\sigma}{2},\label{eq:2.54}\\
 \int_{-L/2}^{L/2}D_{q_n}
 &<4\delta\int_{-L/2}^{L/2}F_{q_n},\label{eq:2.55}
\end{align}
and no $\phi\in C_c^\infty((-(n+\tfrac12)L, (n+\tfrac12)L),\C)$ satisfies
\begin{equation}\label{eq:2.56}
 \|\phi\|_2=1,\qquad \|\phi'\|_2\le n,\qquad
 Q_{q_n,c_n}(\phi)\le-\frac4n.
\end{equation}

After taking a subsequence, $c_n\to c\in J$ and the bounded initial data $(q_n(0),q_n'(0))$ converge. The equation
\begin{equation*}
 q_n''+ic_nq_n'+(1-|q_n|^2)q_n=0
\end{equation*}
and the state bound yield, by Arzelà--Ascoli on every fixed interval,
\begin{equation}\label{eq:2.57}
 q_n\longrightarrow q
 \quad\text{locally together with all derivatives},
\end{equation}
where $R_c(q)=0$ on $\R$ and $\sup_\R|(q,q')|\le2M_J$.

If $q$ does not satisfy \eqref{eq:2.14}, by Proposition~\ref{prop:2.2} there is a compactly supported $\phi\ne0$ with $Q_{q,c}(\phi)<0$. After dividing by $\|\phi\|_2$, we may assume $\|\phi\|_2=1$. For all large $n$, its support lies in the interval in \eqref{eq:2.56}, $\|\phi'\|_2\le n$, and we deduce from \eqref{eq:2.57} that
\begin{equation*}
 Q_{q_n,c_n}(\phi)\le\frac12Q_{q,c}(\phi)\le-\frac4n,
\end{equation*}
contradicting \eqref{eq:2.56}. Hence $q$ satisfies \eqref{eq:2.14}.

If $F_q\equiv0$, \eqref{eq:2.4} implies $q'=0$ and $|q|=1$. A constant rotation then changes $q$ to $1$, and \eqref{eq:2.57} contradicts \eqref{eq:2.54}. If $F_q$ is not identically zero, its integral over $[-L/2,L/2]$ is positive because a function satisfying \eqref{eq:2.14} has constant $F_q$. Lemma~\ref{lem:2.3} gives
\begin{equation*}
 \int_{-L/2}^{L/2}D_q-4\delta\int_{-L/2}^{L/2}F_q
 \ge(\delta_J-4\delta)\int_{-L/2}^{L/2}F_q>0.
\end{equation*}
By local convergence, $\int_{-L/2}^{L/2}D_{q_n}>4\delta\int_{-L/2}^{L/2}F_{q_n}$ for all large $n$, contradicting \eqref{eq:2.55}.
\end{proof}

\begin{lemma}
\label{lem:2.8}
Fix $K>1$, $\sigma>0$, $0<\delta<\delta_J/16$, and $\nu>0$, and let $M$, $\kappa$, and $C_0$ be supplied by Lemma~\ref{lem:2.7} for these parameters. There is $\varepsilon>0$, depending on the displayed parameters and on $J,K$, with the following property. Let $c\in J$ and $j\in\mathbb Z$, and let $q$ satisfy \eqref{eq:2.32}. If
\begin{equation*}
 f_j\ge\nu,\qquad \mathcal R_j\le\varepsilon f_j,\qquad
 \max(f_{j-1},f_{j+1})<Kf_j,
\end{equation*}
then either
\begin{equation}\label{eq:2.58}
 d_j\ge\delta f_j,
\end{equation}
or there is a test $\phi\in C_c^\infty(J_j,\C)$ satisfying
\begin{equation}\label{eq:2.59}
 \|\phi\|_2=1,\qquad \|\phi'\|_2\le C_0,\qquad
 Q_{q,c}(\phi)\le-2\kappa,
\end{equation}
or, after a constant rotation,
\begin{equation}\label{eq:2.60}
 \|q-1\|_{C^1([(j-1)L,(j+2)L])}<\sigma.
\end{equation}
\end{lemma}

\begin{proof}
Translate the center to zero, and let $q_0$ solve the exact equation with the same Cauchy data:
\begin{equation*}
 q_0(0)=q(0),\qquad q_0'(0)=q'(0),\qquad R_c(q_0)=0.
\end{equation*}
On every part of $J_j\cap[0,\infty)$ on which $|(q_0,q_0')|\le2M_J$, subtract the first-order systems for $(q,q')$ and $(q_0,q_0')$. The uniform bound on these pairs gives
\begin{align}
 &|(q-q_0,q'-q_0')(t)|\notag\\
 &\qquad\le C_J\int_0^t|(q-q_0,q'-q_0')(s)|\dd s
 +\int_0^t|R_c(q)(s)|\dd s.
 \label{eq:2.61}
\end{align}
The same inequality holds to the left of zero after reversing the integration limits. It follows from Gronwall's inequality in \eqref{eq:2.61} that, as long as $|(q_0,q_0')|\le2M_J$,
\begin{equation}\label{eq:2.62}
 \|q-q_0\|_{C^1(J_j)}
 \le C_{J,L,M}\|R_c(q)\|_{L^1(J_j)}
 \le C_{J,L,M}\mathcal R_j^{1/2}.
\end{equation}
By \eqref{eq:2.32}, $|(q,q')|\le M_J$. If $t_*$ were the first point, in either direction from zero, at which $|(q_0(t_*),q_0'(t_*))|=2M_J$, then
\begin{align}
 M_J
 &\le |(q_0(t_*),q_0'(t_*))-(q(t_*),q'(t_*))|\notag\\
 &\le C_{J,L,M}\mathcal R_j^{1/2}
 \le C_{J,L,M}\varepsilon^{1/2}f_j^{1/2}
 \le C_{J,L,M}\varepsilon^{1/2}.
 \label{eq:2.63}
\end{align}
The final step in \eqref{eq:2.63} uses the uniform bound $f_j\le C_{J,L}$. Choose $\varepsilon$ so that the final term in \eqref{eq:2.63} is smaller than $M_J/2$. No such $t_*$ exists, and Lemma~\ref{lem:2.7} applies to $q_0$.

If alternative~1 of Lemma~\ref{lem:2.7} holds for $q_0$, decrease $\varepsilon$ so that $\|q-q_0\|_{C^1(J_j)}<\sigma/2$ in \eqref{eq:2.62}. Together with \eqref{eq:2.53}, this yields \eqref{eq:2.60}. If alternative~3 of Lemma~\ref{lem:2.7} holds for $q_0$, use its test $\phi$ with $\|\phi\|_2=1$ and $Q_{q_0,c}(\phi)\le-4\kappa$. Then
\begin{equation}\label{eq:2.64}
 |Q_{q,c}(\phi)-Q_{q_0,c}(\phi)|
 \le C_J\|q-q_0\|_{L^\infty(J_j)}\|\phi\|_2^2
 \le C_{J,L,M}\mathcal R_j^{1/2}.
\end{equation}
We infer from the universal $C^1$ bound that $f_j\le C_{J,L}$, and $\mathcal R_j^{1/2}\le C_{J,L}\varepsilon^{1/2}$. Choose $\varepsilon$ so that $C_{J,L,M}\mathcal R_j^{1/2}\le2\kappa$ in \eqref{eq:2.64}, which proves \eqref{eq:2.59}.

Suppose alternative~2 of Lemma~\ref{lem:2.7} holds for $q_0$. The maps $(q,q')\mapsto F_q$ and $(q,q')\mapsto D_q$ are Lipschitz on the ball $|(q,q')|\le2M_J$, which contains both $(q,q')$ and $(q_0,q_0')$ on $I_j$. We have
\begin{equation}\label{eq:2.65}
 \left|f_j-\int_{I_j}F_{q_0}\right|
 +\left|d_j-\int_{I_j}D_{q_0}\right|
 \le C_{J,L}\mathcal R_j^{1/2}.
\end{equation}
Because $f_j\ge\nu$ and $\mathcal R_j\le\varepsilon f_j$,
\begin{equation*}
 \mathcal R_j^{1/2}\le(\varepsilon/\nu)^{1/2}f_j.
\end{equation*}
Combining \eqref{eq:2.65} with $\int_{I_j}D_{q_0}\ge4\delta\int_{I_j}F_{q_0}$, we derive
\begin{equation}\label{eq:2.66}
 d_j\ge\left\{4\delta-(4\delta+1)C_{J,L}
 (\varepsilon/\nu)^{1/2}\right\}f_j.
\end{equation}
A final decrease of $\varepsilon$, so that the coefficient in braces is at least $\delta$, converts \eqref{eq:2.66} into \eqref{eq:2.58}.
\end{proof}

\begin{lemma}\label{lem:2.9}
There are constants
\begin{equation*}
 L>1,\quad K>4,\quad M\ge2,\quad \varepsilon>0,\quad
 \delta>0,\quad\kappa>0,\quad C_0<\infty,
\end{equation*}
depending only on $J$, with
\begin{equation*}
 \frac{2}{K-2}<\frac\delta4,
\end{equation*}
such that every $q$ satisfying \eqref{eq:2.32}, every $c\in J$, and every $j\in\mathbb Z$ has at least one of the following properties:
\begin{enumerate}
\item $\mathcal R_j>\varepsilon f_j$.
\item $d_j\ge\delta f_j$.
\item there is $\phi\in H_0^1(J_j,\C)$ such that
\begin{equation}\label{eq:2.67}
 \|\phi\|_{L^2}=1,
 \qquad \|\phi'\|_{L^2}\le C_0,
 \qquad Q_{q,c}(\phi)\le-2\kappa.
\end{equation}
\item
\begin{equation*}
 \max\{f_{j-1},f_{j+1}\}\ge Kf_j.
\end{equation*}
\end{enumerate}
\end{lemma}

\begin{proof}
Choose the constants in the following order. The proof of Lemma~\ref{lem:2.4} provides a positive lower bound for the constant $\delta_0$ in \eqref{eq:2.21}, depending only on $J$ and independently of its growth factor. Fix once and for all
\begin{equation*}
 0<\delta<\min\{\delta_J/128,\delta_0/16\}.
\end{equation*}
Choose $K>4$ so large that $2/(K-2)<\delta/4$, and apply Lemma~\ref{lem:2.4} with growth factor $4K$. This fixes $L$. Apply Lemma~\ref{lem:2.6}. It fixes $\sigma$ and a positive forcing threshold. Notice that its coefficient $\delta_0/2$ in \eqref{eq:2.39} is larger than $\delta$. Apply Lemma~\ref{lem:2.7} with the value of $\delta$ just chosen. This fixes $M,\kappa,C_0$. Choose $\nu>0$ so small that Lemma~\ref{lem:2.5}, on an interval of length $3L$ and with the bound $M_J$, ensures
\begin{equation}\label{eq:2.68}
 \int_{(j-1)L}^{(j+2)L}\bigl(|q'|^2+(1-|q|^2)^2\bigr)
 \le2(1+2K)\nu
 \quad\Longrightarrow\quad
 \inf_{\gamma\in\R}\|q-\e^{i\gamma}\|_{C^1([(j-1)L,(j+2)L])}<\sigma.
\end{equation}
Apply Lemma~\ref{lem:2.8} for this $\nu$. Take $\varepsilon>0$ smaller than the forcing thresholds in that lemma and in Lemma~\ref{lem:2.6}.

Fix a cell and assume that properties 1 and 4 fail. If $f_j<\nu$, then
\begin{equation*}
 \int_{(j-1)L}^{(j+2)L}\bigl(|q'|^2+(1-|q|^2)^2\bigr)
 \le2(f_{j-1}+f_j+f_{j+1})
 <2(1+2K)\nu.
\end{equation*}
By \eqref{eq:2.68}, a constant rotation makes $\|q-1\|_{C^1([(j-1)L,(j+2)L])}<\sigma$. We also have
\begin{equation*}
 \int_{(j-1)L}^{(j+2)L}|R_c(q)|^2\le \mathcal R_j
 \le\varepsilon f_j.
\end{equation*}
Lemma~\ref{lem:2.6} and the assumed failure of property 4 imply $d_j\ge(\delta_0/2)f_j\ge\delta f_j$, which is property 2.

If $f_j\ge\nu$, Lemma~\ref{lem:2.8} gives property 2, property 3, or \eqref{eq:2.60}. In the last case, the same application of Lemma~\ref{lem:2.6} yields property 2. Failure of properties 1 and 4 therefore implies property 2 or property 3. This proves the lemma. Every constant depends only on $J$.
\end{proof}

For later measurability, define the bottom Dirichlet eigenvalue by
\begin{equation}\label{eq:2.69}
 \lambda_1(q,c,J_j):=
 \inf_{\substack{\phi\in H_0^1(J_j,\C)\\\|\phi\|_2=1}}
 Q_{q,c}^{J_j}(\phi).
\end{equation}
The Hilbert space in \eqref{eq:2.69} is real. The form is closed and bounded below, and the embedding $H_0^1(J_j,\C)\hookrightarrow L^2(J_j,\C)$ is compact. Hence the infimum is attained. For every $\phi\in H_0^1(J_j,\C)$ with $\|\phi\|_2=1$,
\begin{equation*}
 Q_{q,c}(\phi)
 \ge\frac12\|\phi'\|_2^2-C_J,
\end{equation*}
by Young's inequality and the universal bound on $q$. If $\lambda_1(q,c,J_j)\le-2\kappa$, the normalized minimizer $\phi$ satisfies
\begin{equation}\label{eq:2.70}
 \frac12\|\phi'\|_2^2-C_J\le\lambda_1\le-2\kappa,
 \qquad
 \|\phi'\|_2\le(2C_J)^{1/2}.
\end{equation}
After increasing $C_0$ to at least $(2C_J)^{1/2}$, \eqref{eq:2.70} shows that property 3 in Lemma~\ref{lem:2.9} holds if and only if
\begin{equation}\label{eq:2.71}
 \lambda_1(q,c,J_j)\le-2\kappa.
\end{equation}
If \eqref{eq:2.71} holds, the normalized minimizing eigenfunction satisfies \eqref{eq:2.67}. Conversely, a test satisfying \eqref{eq:2.67} bounds the infimum in \eqref{eq:2.69} by $-2\kappa$. We use this minimizing eigenfunction when comparing nearby slices in Lemma~\ref{lem:3.2}.

The four properties in Lemma~\ref{lem:2.9} can overlap. We use their displayed order and define only one assignment symbol:
\begin{equation}\label{eq:2.72}
 a_j(q):=\min\{k\in\{1,2,3,4\}:\text{property $k$ holds for cell $j$}\}.
\end{equation}

Lemma~\ref{lem:2.9} yields a local alternative on every cell. We now sum these alternatives over all cells. The fourth alternative is reduced to the first three by a geometric-series estimate.

\begin{proposition}
\label{prop:2.10}
Let $c\in J$, and let $q$ satisfy \eqref{eq:2.32}. Assume that $F_q\in L^1(\R)$ and $R_c(q)\in L^2(\R)$. Then there is a constant $C_J>0$ such that
\begin{equation}\label{eq:2.73}
 \begin{aligned}
 \int_\R F_q\dd x
 &\le
 \frac{1+2/(K-2)}{\delta-2/(K-2)}
 \int_\R D_q\dd x\\
&\qquad +C_J\left\{
 \int_\R|R_c(q)|^2\dd x 
+\#\{j\in\mathbb Z:a_j(q)=3\}\right\}.
 \end{aligned}
\end{equation}
Here $\#$ denotes cardinality and is allowed to be $+\infty$.
\end{proposition}

\begin{proof}
Unless otherwise specified, sums over $j$ run over $\mathbb Z$. Since $|D_q|\le F_q$ and $F_q\in L^1(\R)$,
\begin{equation*}
 \sum_j f_j=\int_\R F_q\dd x,
 \qquad
 \sum_j d_j=\int_\R D_q\dd x.
\end{equation*}

We first estimate the contribution from the indices with $a_j(q)=4$. For such an index, $f_j>0$. If $f_j=0$, then $F_q=0$ on $I_j$ and $d_j=0$, so property $2$ holds and $a_j(q)\le2$.

For each $j$ with $a_j(q)=4$, fix a choice $T(j)\in\{j-1,j+1\}$ such that
\begin{equation*}
 f_{T(j)}\ge Kf_j.
\end{equation*}
Starting from $j$, repeat this choice only while the current index has type $4$. Here $T^0(j)=j$ and $T^{r+1}(j)=T(T^r(j))$ whenever $a_{T^r(j)}(q)=4$. If
\begin{equation*}
 a_{T^r(j)}(q)=4
 \qquad(0\le r<m),
\end{equation*}
then iteration gives
\begin{equation}\label{eq:2.74}
 f_{T^m(j)}\ge K^mf_j.
\end{equation}
The sequence $j,T(j),T^2(j),\ldots$ cannot repeat an index. Indeed, if $T^r(j)=T^s(j)$ for some $0\le r<s$, then \eqref{eq:2.74}, starting at $T^r(j)$, implies
\begin{equation*}
 f_{T^r(j)}=f_{T^s(j)}
 \ge K^{s-r}f_{T^r(j)}>f_{T^r(j)},
\end{equation*}
because $f_{T^r(j)}>0$ and $K>1$, a contradiction. Nor can all the indices have type $4$: since they are distinct, for every $m$ we would have
\begin{equation*}
 \sum_{k\in\mathbb Z}f_k
 \ge\sum_{r=0}^m f_{T^r(j)}
 \ge f_j\sum_{r=0}^m K^r\longrightarrow\infty
 \qquad(m\to\infty),
\end{equation*}
contrary to $\sum_k f_k<\infty$. Thus the rule stops after finitely many steps at an index $\ell$ with $a_\ell(q)\in\{1,2,3\}$.

Fix an index $\ell$ with $a_\ell(q)\in\{1,2,3\}$. Since $T(j)\in\{j-1,j+1\}$, the equality $T(j)=\ell$ requires $j=\ell-1$ or $j=\ell+1$. At every preceding step there are again at most two choices. Hence at most $2^n$ initial indices reach $\ell$ after exactly $n$ steps. For each of them, we infer from \eqref{eq:2.74} that
\begin{equation*}
 f_j\le K^{-n}f_\ell.
\end{equation*}
Consequently,
\begin{equation*}
 \sum_{\substack{a_j(q)=4\\
                  j\text{ reaches }\ell}}f_j
 \le\sum_{n=1}^\infty\left(\frac2K\right)^nf_\ell
 =\frac2{K-2}f_\ell.
\end{equation*}
Because the choices $T(j)$ are fixed, each $j$ of type $4$ contributes to the sum for exactly one $\ell$. Summing over all $\ell$ with $a_\ell(q)\in\{1,2,3\}$, we obtain
\begin{equation}\label{eq:2.75}
 \sum_{a_j(q)=4}f_j
 \le\frac2{K-2}
 \sum_{a_j(q)\in\{1,2,3\}}f_j.
\end{equation}

We next estimate the contributions from types $1$ and $3$. The uniform bound \eqref{eq:2.32} and the fixed length $L$ give $f_j\le C_J$. Therefore
\begin{equation*}
 \sum_{a_j(q)=3}f_j
 \le C_J\#\{j\in\mathbb Z:a_j(q)=3\}.
\end{equation*}
If $a_j(q)=1$, then property $1$ implies $f_j<\varepsilon^{-1}\mathcal R_j$. Since the intervals $J_j$ have overlap multiplicity $2M+1$,
\begin{equation*}
 \sum_{a_j(q)=1}f_j
 \le\varepsilon^{-1}\sum_{a_j(q)=1}\mathcal R_j
 \le\frac{2M+1}{\varepsilon}
 \int_\R|R_c(q)|^2\dd x.
\end{equation*}
It follows that
\begin{equation*}
 \sum_{a_j(q)\in\{1,3\}}f_j
 \le C_J\left\{
 \int_\R|R_c(q)|^2\dd x
 +\#\{j\in\mathbb Z:a_j(q)=3\}\right\}.
\end{equation*}

For the indices of type $2$, we deduce from property $2$ and $|D_q|\le F_q$ that
\begin{equation*}
 \begin{aligned}
 \delta\sum_{a_j(q)=2}f_j
 \le\sum_{a_j(q)=2}d_j =\int_\R D_q\dd x-\sum_{a_j(q)\ne2}d_j \le\int_\R D_q\dd x
 +\sum_{a_j(q)\in\{1,3,4\}}f_j.
 \end{aligned}
\end{equation*}
From \eqref{eq:2.75}, we conclude that
\begin{equation*}
 \begin{aligned}
 \left(\delta-\frac2{K-2}\right)
 \sum_{a_j(q)=2}f_j
 &\le\int_\R D_q\dd x +\left(1+\frac2{K-2}\right)
 \sum_{a_j(q)\in\{1,3\}}f_j.
 \end{aligned}
\end{equation*}
The coefficient on the left is positive because $2/(K-2)<\delta/4$.

A second application of \eqref{eq:2.75} yields
\begin{equation*}
 \begin{aligned}
 \sum_jf_j
 &\le\left(1+\frac2{K-2}\right)
 \left(
 \sum_{a_j(q)\in\{1,3\}}f_j
 +\sum_{a_j(q)=2}f_j
 \right)\\
 &\le
 \frac{1+2/(K-2)}{\delta-2/(K-2)}
 \int_\R D_q\dd x\\
 &\quad+
 \left(1+\frac2{K-2}\right)
 \left(
 1+\frac{1+2/(K-2)}{\delta-2/(K-2)}
 \right)
 \sum_{a_j(q)\in\{1,3\}}f_j.
 \end{aligned}
\end{equation*}
We retain the coefficient of $\int_\R D_q$, since this integral is not assumed to be nonnegative. The coefficient of the last sum depends only on $J$. Substituting the preceding estimate for the type $1$ and type $3$ indices proves \eqref{eq:2.73}.
\end{proof}

\section{The two-dimensional energy estimate and existence theory}

In this section $\psi$ is a finite-energy solution of \eqref{eq:1.2}, $c$ belongs to a fixed compact interval $J\subset(0,\sqrt2)$, and $q_y$ is the function defined by \eqref{eq:2.1}. We first estimate the residual contribution $\int_\R|R_c(q_y)|^2\dd x$ and the count $\#\{j:a_j(q_y)=3\}$ after integration over $y$. We then integrate that estimate with respect to $y$ and prove Theorems~\ref{thm:1.1} and~\ref{thm:1.2}.

\subsection{Planar identities and the transverse derivative}

Define two functions on $\R^2$ by applying \eqref{eq:2.4} on each horizontal line:
\begin{equation}\label{eq:3.1}
 \begin{aligned}
 F(\mathbf{x})&:=F_{q_y}(x)
 =|\partial_x\psi|^2+\frac12(1-|\psi|^2)^2,\\
 D(\mathbf{x})&:=D_{q_y}(x)
 =|\partial_x\psi|^2-\frac12(1-|\psi|^2)^2.
 \end{aligned}
\end{equation}
In particular, $F\ge0$ and $|D|\le F$ pointwise. For $d=2$, \cite[Equation~(2.1)]{BellazziniRuiz} and \cite[Lemma~2.2]{BellazziniRuiz} give the first two identities in \eqref{eq:3.2}. Combining them, we obtain the third identity.
\begin{equation}\label{eq:3.2}
 I_c(\psi)=\int_{\R^2}|\partial_y\psi|^2\dd\mathbf{x},
 \qquad
 cP(\psi)=\frac12\int_{\R^2}(1-|\psi|^2)^2\dd\mathbf{x},
 \qquad
 \int_{\R^2}D\dd\mathbf{x}
 =\int_{\R^2}|\partial_y\psi|^2\dd\mathbf{x}.
\end{equation}
We verify the last identity, which was not explicitly given in \cite{BellazziniRuiz}. From the definition of $I_c$, we have
\begin{equation}\label{eq:3.3}
 I_c(\psi)
 =\frac12\int_{\R^2}
 \bigl(|\partial_x\psi|^2+|\partial_y\psi|^2\bigr)\dd\mathbf{x}
 +\frac14\int_{\R^2}(1-|\psi|^2)^2\dd\mathbf{x}
 -cP(\psi).
\end{equation}
Substituting the first two identities in \eqref{eq:3.2} into \eqref{eq:3.3} yields
\begin{equation}\label{eq:3.4}
 \int_{\R^2}|\partial_y\psi|^2\dd\mathbf{x}
 =\frac12\int_{\R^2}
 \bigl(|\partial_x\psi|^2+|\partial_y\psi|^2\bigr)\dd\mathbf{x}
 -\frac14\int_{\R^2}(1-|\psi|^2)^2\dd\mathbf{x}.
\end{equation}
Rearranging \eqref{eq:3.4}, we derive
\begin{equation}\label{eq:3.5}
 \int_{\R^2}
 \left\{|\partial_x\psi|^2
 -\frac12(1-|\psi|^2)^2\right\}\dd\mathbf{x}
 =\int_{\R^2}|\partial_y\psi|^2\dd\mathbf{x},
\end{equation}
which is the third identity in \eqref{eq:3.2}.

It follows from the first equality in \eqref{eq:3.2} and the definition of $F$ that
\begin{equation}\label{eq:3.6}
 E(\psi)=\frac12\int_{\R^2}
 \bigl(F+|\partial_y\psi|^2\bigr)\dd\mathbf{x}.
\end{equation}

We next control the right-hand side of \eqref{eq:2.2} by the same integral $I_c(\psi)$.
\begin{lemma}\label{lem:3.1}
There is $C_J$ such that
\begin{equation}\label{eq:3.7}
 \int_{\R^2}|\nabla\partial_y\psi|^2\dd\mathbf{x}
 \le C_J\int_{\R^2}|\partial_y\psi|^2\dd\mathbf{x}.
\end{equation}
Consequently,
\begin{equation}\label{eq:3.8}
 \int_{\R^2}|\partial_{yy}\psi|^2\dd\mathbf{x}
 \le C_J I_c(\psi).
\end{equation}
\end{lemma}

\begin{proof}
Put $w=\partial_y\psi$. Differentiating \eqref{eq:1.2} with respect to $y$, we obtain
\begin{equation}\label{eq:3.9}
 ic\partial_xw+\Delta w+(1-|\psi|^2)w
 -2\Rea(\overline\psi w)\psi=0.
\end{equation}
Choose $\chi\in C_c^\infty(\R^2)$ with $0\le\chi\le1$, $\chi=1$ on $B_1$, and $\chi=0$ outside $B_2$, and put $\chi_R(\mathbf{x})=\chi(\mathbf{x}/R)$. Taking the real scalar product of \eqref{eq:3.9} with $\chi_R^2w$ and integrating by parts gives
\begin{align}
 \int\chi_R^2|\nabla w|^2
 &=-2\int\chi_R\langle\nabla w,w\rangle\cdot\nabla\chi_R
   +c\int\chi_R^2\langle i\partial_xw,w\rangle\notag\\
 &\quad+\int\chi_R^2(1-|\psi|^2)|w|^2
   -2\int\chi_R^2\langle w,\psi\rangle^2.
 \label{eq:3.10}
\end{align}
All solutions under consideration satisfy a uniform $L^\infty$ bound \cite[Lemma~2.1]{BellazziniRuiz}. By discarding the last nonpositive term in \eqref{eq:3.10} and using Young's inequality, we infer
\begin{align*}
 \int\chi_R^2|\nabla w|^2
 &\le \frac12\int\chi_R^2|\nabla w|^2
 +C_J\int_{B_{2R}}|w|^2
 +\frac{C}{R^2}\int_{B_{2R}\setminus B_R}|w|^2.
\end{align*}
After absorbing the first term on the right, let $R\to\infty$ and use Fatou's lemma. The result is
\begin{equation*}
 \|\nabla w\|_2^2\le C_J\|w\|_2^2,
\end{equation*}
which is \eqref{eq:3.7}. Finally, \eqref{eq:3.2} implies $\|w\|_2^2=I_c(\psi)$. Combining $|\partial_{yy}\psi|\le|\nabla w|$ with \eqref{eq:3.7} proves \eqref{eq:3.8}.
\end{proof}

\subsection{Control of the one-dimensional negative directions}

All entire solutions with $c\in J$ have uniform $C^k$ bounds by the same universal $L^\infty$ estimate and elliptic regularity \cite[Lemma~2.1]{BellazziniRuiz}. We now compare the one-dimensional tests for the functions $q_y$ as $y$ varies. For fixed $j$, let $f_j(y)$, $d_j(y)$, and $\mathcal R_j(y)$ denote the quantities in \eqref{eq:2.33} evaluated at $q=q_y$. These functions are continuous, because their integrands depend continuously on $(x,y)$ and the integration intervals are bounded. The same holds for $f_{j-1}(y)$ and $f_{j+1}(y)$. The map $y\mapsto q_y$ is also continuous from $\R$ to $C^0(J_j)$. For $y,z\in\R$ and $\phi\in H_0^1(J_j,\C)$ with $\|\phi\|_{L^2(J_j)}=1$,
\begin{equation}\label{eq:3.11}
 |Q_{q_y,c}(\phi)-Q_{q_z,c}(\phi)|
 \le C_J\|q_y-q_z\|_{L^\infty(J_j)}.
\end{equation}
Taking the infimum in \eqref{eq:2.69} in both directions proves that $y\mapsto\lambda_1(q_y,c,J_j)$ is continuous.

Define the indicator
\begin{equation}\label{eq:3.12}
 b_j(y):=
 \begin{cases}
 1,&\text{if $a_j(q_y)=3$},\\
 0,&\text{otherwise}.
 \end{cases}
\end{equation}
From equations \eqref{eq:2.71} and \eqref{eq:2.72}, we obtain the explicit identity
\begin{equation}\label{eq:3.13}
 \{y:b_j(y)=1\}
 =\{y:\mathcal R_j(y)\le\varepsilon f_j(y)\}
 \cap\{y:d_j(y)<\delta f_j(y)\}
 \cap\{y:\lambda_1(q_y,c,J_j)\le-2\kappa\}.
\end{equation}
The continuity of $f_j(y)$, $d_j(y)$, $\mathcal R_j(y)$, and $\lambda_1(q_y,c,J_j)$ makes each set on the right of \eqref{eq:3.13} Borel. Hence $b_j$ is measurable.

\begin{lemma}\label{lem:3.2}
There are $H>1$ and $\tau>0$, depending only on $J$, with the following property. Let $j\in\mathbb Z$ and $y_0\in\R$. Suppose $b_j(y_0)=1$ and
\begin{equation}\label{eq:3.14}
 \int_{y_0-H}^{y_0+H}\int_{J_j}
 |\partial_y\psi(x,y)|^2\dd x\dd y\le\tau.
\end{equation}
Then there is $\Phi\in C_c^\infty(J_j\times(y_0-H,y_0+H),\C)$ such that
\begin{equation}\label{eq:3.15}
 Q_{\psi,c}(\Phi)<0.
\end{equation}
\end{lemma}

\begin{proof}
Assume \eqref{eq:3.14}, and choose the normalized minimizing Dirichlet test $\phi$ for $Q_{q_{y_0},c}$, which satisfies \eqref{eq:2.67} with $q=q_{y_0}$. For $|y-y_0|\le H$, only the zeroth-order coefficients involving $\psi$ change with $y$, and their difference is bounded by $C_J|q_y-q_{y_0}|$. We have
\begin{align}
 &|Q_{q_y,c}(\phi)-Q_{q_{y_0},c}(\phi)| \le
 C_J\|q_y-q_{y_0}\|_{L^2(J_j)}\|\phi\|_{L^4(J_j)}^2
 \le C_J\|q_y-q_{y_0}\|_{L^2(J_j)}.
 \label{eq:3.16}
\end{align}
The second inequality in \eqref{eq:3.16} uses
\begin{equation*}
 \|\phi\|_4^2\le
 C\|\phi\|_2(\|\phi'\|_2+\|\phi\|_2)\le C_J.
\end{equation*}
Also,
\begin{equation*}
 q_y(x)-q_{y_0}(x)
 =\int_{y_0}^y\partial_y\psi(x,s)\dd s,
\end{equation*}
so Cauchy--Schwarz yields
\begin{equation}\label{eq:3.17}
 \|q_y-q_{y_0}\|_{L^2(J_j)}
 \le\left(
 |y-y_0|\int_{y_0-H}^{y_0+H}\int_{J_j}
 |\partial_y\psi|^2
 \right)^{1/2}.
\end{equation}

Choose a fixed nonzero $\chi_0\in C_c^\infty(-1,1)$ and take $H$ so large that
\begin{equation*}
 \int\left|\frac1H\chi_0'\!\left(\frac{y-y_0}{H}\right)\right|^2\dd y
 \le\frac\kappa4\int\left|\chi_0\!\left(
 \frac{y-y_0}{H}\right)\right|^2\dd y.
\end{equation*}
The ratio of the derivative integral to the undifferentiated integral is $H^{-2}\|\chi_0'\|_2^2/\|\chi_0\|_2^2$, so this choice of $H$ is possible. Next choose $\tau$ so small that $C_J(H\tau)^{1/2}\le\kappa/2$. Equations \eqref{eq:3.16}--\eqref{eq:3.17} and \eqref{eq:2.67} then imply $Q_{q_y,c}(\phi)\le-3\kappa/2$ whenever $\chi_0((y-y_0)/H)\ne0$. For $\Phi(x,y)=\phi(x)\chi_0((y-y_0)/H)$, it follows from Fubini's theorem that
\begin{align*}
 Q_{\psi,c}(\Phi)
 &=\int\chi_0\!\left(\frac{y-y_0}{H}\right)^2
 Q_{q_y,c}(\phi)\dd y\\
 &\quad+\int\left|\frac1H\chi_0'\!\left(
 \frac{y-y_0}{H}\right)\right|^2\dd y
 \int|\phi|^2\dd x\\
 &\le-\frac{5\kappa}{4}\int\chi_0\!\left(
 \frac{y-y_0}{H}\right)^2\dd y<0.
\end{align*}
If the minimizing Dirichlet test is initially only in $H_0^1(J_j,\C)$, choose $\phi_n\in C_c^\infty(J_j,\C)$ with $\phi_n\to\phi$ in $H^1$. Then
\begin{equation*}
 Q_{\psi,c}\!\left(\phi_n(x)\chi_0\!\left(
 \frac{y-y_0}{H}\right)\right)\longrightarrow
 Q_{\psi,c}(\Phi)<0,
\end{equation*}
so the displayed product with $\phi_n$ is the smooth test required in \eqref{eq:3.15} for all sufficiently large $n$.
\end{proof}

\begin{proposition}\label{prop:3.3}
If $\ind(\psi)<\infty$, then
\begin{equation}\label{eq:3.18}
 \int_\R\sum_{j\in\mathbb Z}b_j(y)\dd y
 \le C_J\left(\ind(\psi)+
 \int_{\R^2}|\partial_y\psi|^2\dd\mathbf{x}\right).
\end{equation}
\end{proposition}

\begin{proof}
First consider the pairs $(j,y)$ for which $b_j(y)=1$ and
\begin{equation}\label{eq:3.19}
 \int_{y-H}^{y+H}\int_{J_j}|\partial_y\psi(x,s)|^2
 \dd x\dd s>\tau.
\end{equation}
By Chebyshev's inequality and Fubini's theorem, we derive
\begin{align}
 &\sum_j\left|\left\{y:b_j(y)=1,\
 \eqref{eq:3.19}\text{ holds}\right\}\right|\notag\\
 &\quad\le\frac1\tau\sum_j\int_\R
 \left\{\int_{y-H}^{y+H}\int_{J_j}|\partial_y\psi(x,s)|^2
 \dd x\dd s\right\}\dd y\notag\\
 &=\frac{2H}{\tau}\sum_j\int_{J_j\times\R}
 |\partial_y\psi|^2\dd\mathbf{x}\notag\\
 &\le\frac{2H(2M+1)}{\tau}\int_{\R^2}
 |\partial_y\psi|^2\dd\mathbf{x}.
 \label{eq:3.20}
\end{align}

Consider the pairs $(j,y)$ satisfying
\begin{equation*}
 b_j(y)=1,
 \qquad
 \int_{y-H}^{y+H}\int_{J_j}|\partial_y\psi(x,s)|^2
 \dd x\dd s\le\tau.
\end{equation*}
From these pairs $(j,y)$, choose rectangles $J_j\times(y-H,y+H)$ successively, requiring each new rectangle to be disjoint from all those already chosen. Stop when no further disjoint rectangle can be added. This process stops after at most $\ind(\psi)$ choices. Otherwise Lemma~\ref{lem:3.2} would provide $\ind(\psi)+1$ tests with pairwise disjoint supports, and their span would be a negative subspace. We therefore obtain a finite maximal pairwise disjoint family
\begin{equation*}
 J_{j_\ell}\times(y_\ell-H,y_\ell+H),
 \qquad 1\le\ell\le N,
\end{equation*}
of the corresponding rectangles. By Lemma~\ref{lem:3.2}, for each $\ell$ there exists a nonzero test function $\Phi_\ell$, supported in this rectangle, such that $Q_{\psi,c}(\Phi_\ell)<0$. Since these supports are pairwise disjoint,
\begin{equation*}
 Q_{\psi,c}\left(\sum_{\ell=1}^N a_\ell\Phi_\ell\right)
 =\sum_{\ell=1}^N a_\ell^2Q_{\psi,c}(\Phi_\ell)<0
 \qquad
 \text{for every }(a_1,\ldots,a_N)\in\R^N\setminus\{0\}.
\end{equation*}
Thus
\begin{equation*}
 N\le\ind(\psi).
\end{equation*}
By maximality, every rectangle associated with a pair under consideration intersects one of the $N$ selected rectangles. Such an intersection implies
\begin{equation*}
 |j-j_\ell|\le2M,
 \qquad
 |y-y_\ell|<2H
\end{equation*}
for some $\ell$. For each selected rectangle, there are at most $4M+1$ possible integers $j$, and the possible values of $y$ lie in an interval of length $4H$. Since $N\le\ind(\psi)$, we conclude that
\begin{equation}\label{eq:3.21}
 \sum_j\left|\left\{y:b_j(y)=1,\
 \int_{y-H}^{y+H}\int_{J_j}|\partial_y\psi(x,s)|^2
 \dd x\dd s\le\tau\right\}\right|
 \le4H(4M+1)\ind(\psi).
\end{equation}
Combining \eqref{eq:3.20} and \eqref{eq:3.21} proves \eqref{eq:3.18}.
\end{proof}

\subsection{Integration over horizontal lines and proof of
Theorem~\ref{thm:1.1}}

\begin{proposition}
\label{prop:3.4}
For every compact interval $J\subset(0,\sqrt2)$ there is $C_J$ such that every finite-energy solution $\psi$ of \eqref{eq:1.2} with $c\in J$ satisfies
\begin{equation}\label{eq:3.22}
 \int_{\R^2}F\dd\mathbf{x}
 \le C_J\left\{
 \int_{\R^2}D\dd\mathbf{x}
 +\int_{\R^2}|\partial_{yy}\psi|^2\dd\mathbf{x}
 +\ind(\psi)\right\}.
\end{equation}
\end{proposition}

\begin{proof}
If $\ind(\psi)=\infty$, the right-hand side of \eqref{eq:3.22} is infinite. Hence assume $\ind(\psi)<\infty$. By finite energy, Lemma~\ref{lem:3.1}, and Fubini's theorem, for almost every $y\in\R$,
\begin{equation}\label{eq:3.23}
 \int_\R F_{q_y}\dd x<\infty,
 \qquad
 \int_\R|\partial_{yy}\psi(x,y)|^2\dd x<\infty.
\end{equation}
Here the first assertion follows directly from
\begin{equation}\label{eq:3.24}
 \int_{\R^2}F\dd\mathbf{x}\le2E(\psi)<\infty,
\end{equation}
and the second is a consequence of \eqref{eq:3.7}.

Multiplication by a constant of modulus one does not change any quantity in this proof, so normalize the limit of $\psi$ at infinity to be $1$. The precise estimates from \cite[Theorem~11 and Proposition~33]{GravejatDecay}, specialized to two dimensions, state that there is a constant $C_\psi>0$ such that
\begin{equation}\label{eq:3.25}
 |\psi(\mathbf{x})-1|\le\frac{C_\psi}{|\mathbf{x}|},
 \qquad
 |\nabla\psi(\mathbf{x})|\le\frac{C_\psi}{|\mathbf{x}|^2}
 \qquad(|\mathbf{x}|\ge1).
\end{equation}
For a fixed $y$, $|\mathbf{x}|=(x^2+y^2)^{1/2}\to\infty$ as $x\to\pm\infty$. Hence, we conclude from \eqref{eq:2.1} and \eqref{eq:3.25} that
\begin{equation*}
 \lim_{x\to\pm\infty}q_y(x)=1,
 \qquad
 \lim_{x\to\pm\infty}q_y'(x)=0.
\end{equation*}
The uniform regularity estimates imply \eqref{eq:2.32}. Therefore Proposition~\ref{prop:2.10} applies for every $y$ satisfying \eqref{eq:3.23}.

We integrate \eqref{eq:2.73} with respect to $y$. First, \eqref{eq:3.1} gives $|D|\le F$. It follows from \eqref{eq:3.24} that
\begin{equation}\label{eq:3.26}
 \int_{\R^2}|D|\dd\mathbf{x}<\infty,
 \qquad
 \int_\R\left\{\int_\R D_{q_y}(x)\dd x\right\}\dd y
 =\int_{\R^2}D(\mathbf{x})\dd\mathbf{x}.
\end{equation}
Second, we infer from \eqref{eq:2.5} and Tonelli's theorem that
\begin{equation*}
 \int_\R\int_\R|R_c(q_y)(x)|^2\dd x\dd y
 =\int_{\R^2}|\partial_{yy}\psi(\mathbf{x})|^2\dd\mathbf{x}.
\end{equation*}
Third, since $b_j(y)$ is the indicator defined in \eqref{eq:3.12},
\begin{equation}\label{eq:3.27}
 \int_\R\#\{j:b_j(y)=1\}\dd y
 =\int_\R\sum_{j\in\mathbb Z}b_j(y)\dd y.
\end{equation}
Integrating \eqref{eq:2.73} and using \eqref{eq:3.26}--\eqref{eq:3.27}, we obtain
\begin{align}
 \int_{\R^2}F\dd\mathbf{x}
 &\le \frac{1+2/(K-2)}{\delta-2/(K-2)}
 \int_{\R^2}D\dd\mathbf{x}
 {}+C_J\left\{
 \int_{\R^2}|\partial_{yy}\psi|^2\dd\mathbf{x}+
 \int_\R\sum_jb_j(y)\dd y\right\}.
 \label{eq:3.28}
\end{align}
Proposition~\ref{prop:3.3}, together with the first equality in \eqref{eq:3.2}, yields
\begin{equation}\label{eq:3.29}
 \int_\R\sum_jb_j(y)\dd y
 \le C_J\left(\ind(\psi)+
 \int_{\R^2}|\partial_y\psi|^2\dd\mathbf{x}\right).
\end{equation}
Substitute \eqref{eq:3.29} into \eqref{eq:3.28}. Then use \eqref{eq:3.2}, namely $\int_{\R^2}D\dd\mathbf{x} =\int_{\R^2}|\partial_y\psi|^2\dd\mathbf{x}\ge0$. The resulting inequality is exactly \eqref{eq:3.22}.
\end{proof}

\begin{proof}[Proof of Theorem~\ref{thm:1.1}]
By \eqref{eq:3.2} and Lemma~\ref{lem:3.1},
\begin{equation*}
 \int_{\R^2}D\dd\mathbf{x}
 =\int_{\R^2}|\partial_y\psi|^2\dd\mathbf{x}=I_c(\psi),
 \qquad
 \int_{\R^2}|\partial_{yy}\psi|^2\dd\mathbf{x}\le C_JI_c(\psi).
\end{equation*}
By Proposition~\ref{prop:3.4},
\begin{equation}\label{eq:3.30}
 \int_{\R^2}F\dd\mathbf{x}
 \le C_J\bigl(I_c(\psi)+\ind(\psi)\bigr).
\end{equation}
We conclude from equations \eqref{eq:3.30} and \eqref{eq:3.6} that \eqref{eq:1.6} holds.
\end{proof}

\subsection{Proof of Theorem~\ref{thm:1.2}}

We also prove here the upper bound in \eqref{eq:1.17}. Fix a compact interval $J\subset(0,\sqrt2)$ and choose a compact interval $\widehat J\subset(0,\sqrt2)$ whose interior contains $J$. Given $c\in J$, choose $c_n\to c$ from the full-measure set of speeds in \cite[Theorem~1.1]{BellazziniRuiz}, with every $c_n\in\widehat J$, and let $u_n$ be the corresponding waves. The uniform estimate in that theorem, applied on the fixed interval $\widehat J$, provides a constant $A_J$ such that
\begin{equation}\label{eq:3.31}
 0<I_{c_n}(u_n)\le A_J,
 \qquad \ind(u_n)\le1.
\end{equation}
Theorem~\ref{thm:1.1}, again on $\widehat J$, implies that $E(u_n)$ is bounded. By \cite[Proposition~6.1]{BellazziniRuiz}, after translations and passage to a subsequence, we relabel the transformed sequence as $u_n$ and obtain
\begin{equation*}
 u_n\longrightarrow u
 \qquad\hbox{locally together with all derivatives},
\end{equation*}
where $u$ is a solution at speed $c$. Fatou's lemma implies $E(u)<\infty$. The argument in \cite[Proposition~6.1]{BellazziniRuiz} also shows that $|u(0)|\ne1$. A constant solution is either zero or has modulus one, and the constant zero has infinite energy. Hence $u$ is nonconstant.

To verify the index bound, suppose instead that $\ind(u)\ge2$. There is a two-dimensional real subspace $Y\subset C_c^\infty(\R^2,\C)$ on which $Q_{u,c}$ is negative definite. Compactness of the $H^1$-unit sphere of $Y$ yields $\eta>0$ such that
\begin{equation}\label{eq:3.32}
 Q_{u,c}(\phi)\le-\eta
 \qquad(\phi\in Y,\ \|\phi\|_{H^1}=1).
\end{equation}
All elements of $Y$ have support in one compact set, so local smooth convergence and $c_n\to c$ imply uniform convergence of the quadratic-form matrices on $Y$. It follows from \eqref{eq:3.32} that, for large $n$, the form $Q_{u_n,c_n}$ is negative on $Y\setminus\{0\}$, contradicting \eqref{eq:3.31}. Hence $\ind(u)\le1$. By \cite[Theorem~9]{BethuelGravejatSautSurvey}, we may multiply $u$ by a constant of modulus one so that it satisfies \eqref{eq: normalize}. We continue to denote the resulting solution by $u$. This operation leaves the index and $\int_{\R^2}|\partial_yu|^2\dd\mathbf{x}$ unchanged. Finally, we infer from the first identity in \eqref{eq:3.2} and Fatou's lemma that
\begin{equation}\label{eq:3.33}
 I_c(u)=\int_{\R^2}|\partial_yu|^2\dd\mathbf{x}
 \le\liminf_{n\to\infty}\int_{\R^2}|\partial_yu_n|^2\dd\mathbf{x}
 =\liminf_{n\to\infty}I_{c_n}(u_n)\le A_J.
\end{equation}
Thus $u$ is a nonconstant finite-energy solution at the prescribed speed $c$, with $\ind(u)\le1$. It is also admissible in \eqref{eq:1.16}, and $\beta_1(c)\le A_J$. Since $c\in J$ was arbitrary, this proves Theorem~\ref{thm:1.2} and the upper bound in \eqref{eq:1.17}.

\section{Compactness and minimization}

\subsection{Finite-bubble compactness}

We prove Theorem~\ref{thm:1.3} by selecting finitely many profiles as in the concentration--compactness principle \cite{LionsConcentration}. The estimate which excludes energy between the selected profiles is proved in Lemma~\ref{lem:4.3}. We use throughout the fact that finite-energy solutions with speed in a fixed compact subset of $(0,\sqrt2)$ have uniform bounds for every derivative. The constants depend only on the derivative order and the compact speed interval, and are independent of the solution and the spatial point. This follows from the universal bound in \cite[Lemma~2.1]{BellazziniRuiz} and local elliptic estimates.

We begin with the small-amplitude estimate used both in the profile selection and in the proof that the minimizing action is bounded away from zero.

B\'ethuel, Gravejat, and Saut \cite[Proposition~2.4]{BethuelGravejatSaut} proved a quantitative modulus gap, and Lemma~\ref{lem:4.1} establishes the uniform modulus gap used in this paper.

\begin{lemma}\label{lem:4.1}
For every compact interval $J\subset(0,\sqrt2)$ there is $\eta_J>0$ such that every finite-energy solution $\psi$ of \eqref{eq:1.2} with $c\in J$ satisfying
\begin{equation}\label{eq:4.1}
 \|1-|\psi|\|_{L^\infty(\R^2)}\le\eta_J,
\end{equation}
is constant of modulus one.
\end{lemma}

\begin{proof}
Choose initially $0<\eta_J<1/2$. Then \eqref{eq:4.1} implies $|\psi|\ge1/2$. Since $\R^2$ is simply connected, there are real-valued functions $s$ and $\theta$ such that
\begin{equation}\label{eq:4.2}
 \psi=(1+s)\e^{i\theta},\qquad
 \|s\|_\infty\le\eta_J,\qquad 1+s\ge\frac12.
\end{equation}
Here $1+s=|\psi|$, while $\theta$ is a real phase. By directly differentiating \eqref{eq:4.2}, we derive
\begin{align*}
 \nabla\psi&=\e^{i\theta}
 \{\nabla s+i(1+s)\nabla\theta\},\\
 \Delta\psi&=\e^{i\theta}\{\Delta s-(1+s)|\nabla\theta|^2
 +i[2\nabla s\cdot\nabla\theta+(1+s)\Delta\theta]\},\\
 ic\partial_x\psi&=\e^{i\theta}
 \{ic\partial_xs-c(1+s)\partial_x\theta\},\\
 (1-|\psi|^2)\psi&=\e^{i\theta}(-2s-3s^2-s^3).
\end{align*}
After dividing \eqref{eq:1.2} by $\e^{i\theta}$, the resulting modulus and phase equations are
\begin{align}
 \Delta s-c(1+s)\partial_x\theta
 -(1+s)|\nabla\theta|^2-2s-3s^2-s^3&=0,
 \label{eq:4.3}\\
 \operatorname{div}\bigl((1+s)^2\nabla\theta\bigr)
 +\frac c2\partial_x(1+s)^2&=0.
 \label{eq:4.4}
\end{align}
The imaginary part is $c\partial_xs+2\nabla s\cdot\nabla\theta+(1+s)\Delta\theta=0$. Multiplying it by $1+s$ gives \eqref{eq:4.4}. The same differentiation also yields the exact energy formula
\begin{equation*}
 e(\psi)=\frac12\{|\nabla s|^2+(1+s)^2|\nabla\theta|^2\}
 +s^2\left(1+\frac{s}{2}\right)^2.
\end{equation*}
Let $\zeta_S\in C_c^\infty(\R^2)$ satisfy
\begin{equation}\label{eq:4.5}
 0\le\zeta_S\le1,\quad
 \zeta_S=1\text{ on }B_S,\quad
 \zeta_S=0\text{ on }\R^2\setminus B_{2S},\quad
 |\nabla\zeta_S|\le C/S.
\end{equation}
Let
\begin{equation*}
 \theta_S=\frac{1}{|B_{2S}\setminus B_S|}
 \int_{B_{2S}\setminus B_S}\theta\dd\mathbf{x}.
\end{equation*}
It follows from the scale-invariant Poincar\'e inequality that
\begin{equation}\label{eq:4.6}
 \int_{B_{2S}\setminus B_S}|\nabla\zeta_S|^2
 |\theta-\theta_S|^2\dd\mathbf{x}
 \le C\int_{B_{2S}\setminus B_S}|\nabla\theta|^2\dd\mathbf{x}.
\end{equation}
Because $|s|\le\eta_J<1/2$, the density $e(\psi)$ is bounded above and below by fixed positive multiples of
\begin{equation}\label{eq:4.7}
 |\nabla s|^2+|\nabla\theta|^2+s^2.
\end{equation}
Finite energy therefore implies
\begin{equation}\label{eq:4.8}
 \int_{B_{2S}\setminus B_S}
 (|\nabla s|^2+|\nabla\theta|^2+s^2)\dd\mathbf{x}\longrightarrow0
 \qquad(S\to\infty).
\end{equation}

We first test \eqref{eq:4.3}. Multiplying \eqref{eq:4.3} by $\zeta_S^2s$ and integrating by parts, we obtain the finite-$S$ identity
\begin{align*}
 &\int\zeta_S^2\bigl(|\nabla s|^2+2s^2
       +cs\partial_x\theta\bigr)\\
 &\quad=-2\int\zeta_Ss\nabla\zeta_S\cdot\nabla s
 -\int\zeta_S^2\bigl\{cs^2\partial_x\theta
 +(1+s)s|\nabla\theta|^2+3s^3+s^4\bigr\}.
\end{align*}
The first term on the right is supported in $B_{2S}\setminus B_S$ and tends to zero by Cauchy--Schwarz and \eqref{eq:4.8}. The remaining integrands are absolutely integrable because $s$ is bounded and \eqref{eq:4.7} is integrable. Thus, letting $S\to\infty$, we conclude that
\begin{align}
 &\int_{\R^2}\bigl(
 |\nabla s|^2+2s^2+cs\partial_x\theta\bigr)\dd\mathbf{x}\notag\\
 &\qquad
 =-\int_{\R^2}\bigl(
 cs^2\partial_x\theta+(1+s)s|\nabla\theta|^2
 +3s^3+s^4\bigr)\dd\mathbf{x}.
 \label{eq:4.9}
\end{align}

For the second equation, note that $\partial_x(1+s)^2=\partial_x\{(1+s)^2-1\}$. Testing \eqref{eq:4.4} with $\zeta_S^2(\theta-\theta_S)$, we therefore derive
\begin{align*}
 &\int\zeta_S^2\left\{(1+s)^2|\nabla\theta|^2
 +\frac c2\bigl((1+s)^2-1\bigr)\partial_x\theta\right\}\\
 &\quad=-2\int\zeta_S(\theta-\theta_S)
 \left\{(1+s)^2\nabla\theta
 +\frac c2\bigl((1+s)^2-1\bigr)(1,0)\right\}
 \cdot\nabla\zeta_S.
\end{align*}
By Cauchy--Schwarz, \eqref{eq:4.6}, and \eqref{eq:4.8}, the right-hand side tends to zero. Since $\frac12\{(1+s)^2-1\}=s+\frac12s^2$, the limit is
\begin{equation}\label{eq:4.10}
 \int_{\R^2}\left(
 (1+s)^2|\nabla\theta|^2
 +cs\partial_x\theta+\frac c2s^2\partial_x\theta
 \right)\dd\mathbf{x}=0.
\end{equation}

Adding \eqref{eq:4.9} and \eqref{eq:4.10}, we infer
\begin{align*}
 &\int_{\R^2}\left\{|\nabla s|^2+(1+s)^2|\nabla\theta|^2
 +2s^2+2cs\partial_x\theta\right\}\dd\mathbf{x}\\
 &\quad=-\int_{\R^2}\left\{
 \frac{3c}{2}s^2\partial_x\theta+(1+s)s|\nabla\theta|^2
 +3s^3+s^4\right\}\dd\mathbf{x}.
\end{align*}
For $|s|\le1/2$ the four terms on the last line satisfy, pointwise,
\begin{align*}
 |s^2\partial_x\theta|&\le
 \frac{|s|}{2}\{s^2+|\partial_x\theta|^2\},\\
 |(1+s)s|\nabla\theta|^2|&\le\frac32|s||\nabla\theta|^2,
 \qquad |3s^3+s^4|\le\frac72|s|s^2.
\end{align*}
Using these inequalities in the last equality, we obtain
\begin{align}
 &\int_{\R^2}\left(
 |\nabla s|^2+(1+s)^2|\nabla\theta|^2
 +2s^2+2cs\partial_x\theta\right)\dd\mathbf{x}\notag\\
 &\qquad\le
 C_J\|s\|_\infty
 \int_{\R^2}(|\nabla s|^2+|\nabla\theta|^2+s^2)\dd\mathbf{x}.
 \label{eq:4.11}
\end{align}

For $c\in J$,
\begin{align}
 2s^2+|\partial_x\theta|^2+2cs\partial_x\theta
 &=\left(\sqrt2s+\frac c{\sqrt2}\partial_x\theta\right)^2
 +\left(1-\frac{c^2}{2}\right)|\partial_x\theta|^2\notag\\
 &\ge\left(1-\frac{\max J}{\sqrt2}\right)
 \left(2s^2+|\partial_x\theta|^2\right).
 \label{eq:4.12}
\end{align}
We also have $|(1+s)^2-1|\le(2+\eta_J)\eta_J$. Hence the left-hand side of \eqref{eq:4.11}, using \eqref{eq:4.12}, is bounded below by
\begin{align*}
 &\min\left\{1,1-\frac{\max J}{\sqrt2}\right\}
 \int_{\R^2}(|\nabla s|^2+|\nabla\theta|^2+s^2)\dd\mathbf{x}\\
 &\qquad-C\eta_J
 \int_{\R^2}|\nabla\theta|^2\dd\mathbf{x}.
\end{align*}
Choose $\eta_J>0$ so small that $(C+C_J)\eta_J\le\frac12\min\{1,1-\max J/\sqrt2\}$, where $C$ and $C_J$ are the constants in the preceding lower bound and in \eqref{eq:4.11}. The two error terms are then absorbed into the positive integral, forcing
\begin{equation*}
 \int_{\R^2}(|\nabla s|^2+|\nabla\theta|^2+s^2)\dd\mathbf{x}=0.
\end{equation*}
Thus $s=0$ and $\nabla\theta=0$, which proves the lemma.
\end{proof}

Tarquini \cite{Tarquini} proved a positive energy lower bound for each fixed speed, and Lemma~\ref{lem:4.2} shows that this lower bound is uniform for $c\in J$.

\begin{lemma}
\label{lem:4.2}
For every compact interval $J\subset(0,\sqrt2)$ there is $e_J>0$ such that every nonconstant finite-energy solution $\psi$ with speed in $J$ satisfies
\begin{equation}\label{eq:4.13}
 E(\psi)\ge e_J.
\end{equation}
\end{lemma}

\begin{proof}
By Lemma~\ref{lem:4.1}, we have the uniform modulus gap
\begin{equation}\label{eq:4.14}
 \|1-|\psi|\|_{L^\infty(\R^2)}>\eta_J
\end{equation}
for every nonconstant wave with $c\in J$. The universal bound in \cite[Lemma~2.1]{BellazziniRuiz} and local elliptic estimates imply, after increasing the constant $M_J$ already used in \eqref{eq:2.32},
\begin{equation*}
 \|\nabla\psi\|_{L^\infty(\R^2)}\le M_J.
\end{equation*}
Choose $\mathbf{x}_0\in\R^2$ with $|1-|\psi(\mathbf{x}_0)||\ge\eta_J/2$. Since $|\nabla|\psi||\le|\nabla\psi|$, for every $\mathbf{x}$ we have
\begin{equation*}
 \bigl||\psi(\mathbf{x})|-|\psi(\mathbf{x}_0)|\bigr|
 \le M_J|\mathbf{x}-\mathbf{x}_0|.
\end{equation*}
Consequently, on $B(\mathbf{x}_0,\eta_J/(4M_J))$ we have
\begin{equation*}
 |1-|\psi(\mathbf{x})||\ge\frac{\eta_J}{4},
 \qquad
 |1-|\psi(\mathbf{x})|^2|\ge\frac{\eta_J}{4}.
\end{equation*}
Therefore
\begin{equation*}
 E(\psi)\ge\frac14\int_{B(\mathbf{x}_0,\eta_J/(4M_J))}
 (1-|\psi|^2)^2\dd\mathbf{x}
 \ge\frac{\pi\eta_J^4}{1024M_J^2}>0,
\end{equation*}
which proves \eqref{eq:4.13}.
\end{proof}

\begin{lemma}
\label{lem:4.3}
Let $J\subset(0,\sqrt2)$ be a compact interval, let $c_n\in J$ satisfy $c_n\to c\in J$, and let $\psi_n$ be normalized solutions of \eqref{eq:1.2} at speed $c_n$ whose energies are uniformly bounded. Suppose that, for a fixed integer $K\ge0$, points $\mathbf{a}_n^1,\ldots,\mathbf{a}_n^K$ satisfy
\begin{equation*}
 |\mathbf{a}_n^k-\mathbf{a}_n^\ell|\longrightarrow\infty\qquad(k\ne\ell),
\end{equation*}
and that every sequence $\mathbf{z}_n$ with $\min_\ell|\mathbf{z}_n-\mathbf{a}_n^\ell|\to\infty$ has only constant functions of modulus one as locally smooth subsequential limits of $\psi_n(\mathbf{z}_n+\cdot)$. When $K=0$, every minimum over $1\le\ell\le K$ is defined to be $+\infty$. Assume also, when $K>0$, that after constant rotations $\psi_n(\mathbf{a}_n^\ell+\cdot)$ converges locally to a normalized finite-energy profile $\psi^\ell$ for every $\ell$. Then
\begin{equation}\label{eq:4.15}
 \lim_{R\to\infty}\limsup_{n\to\infty}
 \int_{\{\mathbf{x}:\min_\ell|\mathbf{x}-\mathbf{a}_n^\ell|>R\}}
 e(\psi_n)\dd\mathbf{x}=0.
\end{equation}
\end{lemma}

\begin{proof}
We first prove that the modulus approaches one uniformly away from the listed centers, namely,
\begin{equation}\label{eq:4.16}
 \lim_{R\to\infty}\limsup_{n\to\infty}
 \sup_{\min_\ell|\mathbf{x}-\mathbf{a}_n^\ell|>R}
 \bigl|1-|\psi_n(\mathbf{x})|\bigr|=0.
\end{equation}
If \eqref{eq:4.16} failed, there would be $\varepsilon_0>0$, radii $R_m\to\infty$, indices $n_m\to\infty$, and points $\mathbf{z}_m$ such that
\begin{equation}\label{eq:4.17}
 \min_\ell|\mathbf{z}_m-\mathbf{a}_{n_m}^\ell|>R_m,\qquad
 |1-|\psi_{n_m}(\mathbf{z}_m)||\ge\varepsilon_0.
\end{equation}
Set $\mathbf{z}_{n_m}=\mathbf{z}_m$, and for the remaining indices choose $\mathbf{z}_n$ so that $\min_\ell|\mathbf{z}_n-\mathbf{a}_n^\ell|\ge n$. The derivative bounds quoted at the start of this section and a diagonal Arzel\`a--Ascoli argument produce a subsequence of $\psi_{n_m}(\mathbf{z}_{n_m}+\cdot)$ that converges on every compact set together with all derivatives. The limit function has modulus different from one at the origin by \eqref{eq:4.17}. This contradicts the hypothesis on the full sequence $\mathbf{z}_n$. This proves \eqref{eq:4.16}.

If $K=0$, \eqref{eq:4.16} says that $|\psi_n|$ is uniformly close to one on the whole plane. For large $n$, \eqref{eq:4.1} holds. By Lemma~\ref{lem:4.1}, we have $E(\psi_n)=0$. Thus \eqref{eq:4.15} holds when $K=0$. We assume $K\ge1$ from now on.

We now construct one real phase on the plane after removing disks of fixed radius around the centers. Since every profile is normalized, its uniform convergence to $1$ at infinity \cite[Theorem~1]{GravejatDecay} provides a common $R_0$ such that
\begin{equation*}
 |\psi^\ell(\mathbf{x})-1|<\frac12
 \qquad(|\mathbf{x}|\ge R_0,\ 1\le\ell\le K).
\end{equation*}
Thus the profiles are nonzero there. Their degree on every circle $\partial B_R$, $R\ge R_0$, is zero because the image of the circle lies in $\{z:|z-1|<1/2\}$ and is therefore homotopic to the constant $1$. Equivalently, with $\partial_\tau$ denoting the counterclockwise tangential derivative and $\dd s$ denoting arclength,
\begin{equation*}
 \deg\left(\frac{\psi^\ell}{|\psi^\ell|},\partial B_R\right)
 =\frac1{2\pi}\int_{\partial B_R}
 \frac{\Ima(\overline{\psi^\ell}\,\partial_\tau\psi^\ell)}
 {|\psi^\ell|^2}\dd s=0.
\end{equation*}
This exterior lifting is also stated in \cite[Lemma~15]{GravejatDecay}. Increase $R_0$ so that $R_0\ge1$, and fix $R\ge R_0$, large enough also that \eqref{eq:4.16} implies $|1-|\psi_n||\le\eta_J$ on the exterior for all sufficiently large $n$. This is the hypothesis used in the polar energy equivalence \eqref{eq:4.7}. The balls $B(\mathbf{a}_n^\ell,2R)$ are disjoint for large $n$. Local convergence transfers nonvanishing and degree zero to $\partial B(\mathbf{a}_n^\ell,3R/2)$. Constant rotations do not change the degree. Equation \eqref{eq:4.16} ensures nonvanishing throughout the exterior domain
\begin{equation*}
 \Omega_{n,R}:=\R^2\setminus
 \bigcup_{\ell=1}^K\overline B(\mathbf{a}_n^\ell,R).
\end{equation*}
Define on this domain the real one-form
\begin{equation*}
 \omega_n:=\frac{\Ima(\overline{\psi_n}\,\dd\psi_n)}{|\psi_n|^2}.
\end{equation*}
Locally, $\omega_n$ is the differential of an argument of $\psi_n$, so it is closed. Its integral around $\partial B(\mathbf{a}_n^\ell,3R/2)$ is $2\pi$ times the degree just computed and is zero. These circles lie in $\Omega_{n,R}$. For a closed one-form on the plane with $K$ disks removed, every integral along a closed curve is an integer linear combination of the integrals around these $K$ circles. Thus the integral of $\omega_n$ along every closed curve in $\Omega_{n,R}$ is zero. The path integral of $\omega_n$ from a fixed point is therefore independent of the path and defines a single-valued real function $\theta_n$ with $\dd\theta_n=\omega_n$. After adding one constant to $\theta_n$, we have
\begin{equation}\label{eq:4.18}
 \psi_n=\rho_n\e^{i\theta_n},
 \qquad \rho_n=|\psi_n|>0
 \qquad\hbox{on }\Omega_{n,R}.
\end{equation}
In particular, $\Ima(\overline{\psi_n}\nabla\psi_n)=\rho_n^2\nabla\theta_n$. Multiplying \eqref{eq:1.2} by $\overline{\psi_n}$ and taking the imaginary part, we obtain
\begin{align}
 &\operatorname{div}\left\{
 \Ima(\overline{\psi_n}\nabla\psi_n)
 +\frac{c_n}{2}(|\psi_n|^2-1)(1,0)\right\}\notag\\
 &=\Ima(\overline{\psi_n}\Delta\psi_n)
 +\frac{c_n}{2}\partial_x|\psi_n|^2\notag\\
 &=-c_n\Rea(\overline{\psi_n}\partial_x\psi_n)
 +\frac{c_n}{2}\partial_x|\psi_n|^2=0.
 \label{eq:4.19}
\end{align}
Choose a fixed smooth radial function $\chi:[0,\infty)\to[0,1]$ with $\chi=0$ on $[0,1]$, $\chi=1$ on $[2,\infty)$, and $|\chi'|\le C$. Define
\begin{equation*}
 \chi_{n,R}(\mathbf{x}):=\prod_{\ell=1}^K
 \chi\!\left(\frac{|\mathbf{x}-\mathbf{a}_n^\ell|}{R}\right).
\end{equation*}
Since the balls $B(\mathbf{a}_n^\ell,2R)$ are disjoint,
\begin{equation}\label{eq:4.20}
 0\le\chi_{n,R}\le1,\qquad
 |\nabla\chi_{n,R}|\le C/R.
\end{equation}
Define the union of inner annuli and their phase means by
\begin{align*}
 A_{n,R}&:=\bigcup_{\ell=1}^K
 \{R<|\mathbf{x}-\mathbf{a}_n^\ell|<2R\},\\
 \theta_{n,\ell}&:=
 \frac{1}{|B_{2R}\setminus B_R|}
 \int_{B(\mathbf{a}_n^\ell,2R)\setminus B(\mathbf{a}_n^\ell,R)}
 \theta_n\dd\mathbf{x}.
\end{align*}
On $B(\mathbf{a}_n^\ell,2R)$ all factors in $\chi_{n,R}$ except the $\ell$th one equal $1$. Hence $\theta_{n,\ell}(1-\chi_{n,R}^2)$ on that ball, extended by zero, is a compactly supported test function. We infer from \eqref{eq:4.19} that
\begin{equation}\label{eq:4.21}
 \int_{B(\mathbf{a}_n^\ell,2R)}
 \left\{\Ima(\overline{\psi_n}\nabla\psi_n)
 +\frac{c_n}{2}(|\psi_n|^2-1)(1,0)\right\}\cdot\nabla
 \{\theta_{n,\ell}(1-\chi_{n,R}^2)\}=0.
\end{equation}
In the flux cutoff term obtained from \eqref{eq:4.19}, identity \eqref{eq:4.21} therefore permits replacing $\theta_n$ by $\theta_n-\theta_{n,\ell}$ on the $\ell$th annulus. From the scaled annular Poincar\'e inequality and \eqref{eq:4.20}, we obtain
\begin{equation}\label{eq:4.22}
 \sum_{\ell=1}^K\int_{\{R<|\mathbf{x}-\mathbf{a}_n^\ell|<2R\}}
 |\nabla\chi_{n,R}|^2|\theta_n-\theta_{n,\ell}|^2
\le C\int_{A_{n,R}}|\nabla\theta_n|^2.
\end{equation}

It remains to estimate the energy in this exterior domain by the energy in the annuli where the cutoff changes. Define
\begin{equation*}
 s_n:=\rho_n-1.
\end{equation*}
In the remainder of this proof, $\|s_n\|_\infty$ abbreviates $\|s_n\|_{L^\infty(\{\chi_{n,R}\ne0\})}$. Equation \eqref{eq:4.16} makes this norm arbitrarily small when $R$ is chosen large and $n$ is then sufficiently large. Substitution of \eqref{eq:4.18} in \eqref{eq:1.2} yields, after division by $\e^{i\theta_n}$,
\begin{align*}
 &\Delta\rho_n-c_n\rho_n\partial_x\theta_n
 -\rho_n|\nabla\theta_n|^2+(1-\rho_n^2)\rho_n\\
 &\quad+i\{c_n\partial_x\rho_n+2\nabla\rho_n\cdot\nabla\theta_n
 +\rho_n\Delta\theta_n\}=0.
\end{align*}
Its real part and its imaginary part multiplied by $\rho_n$ are
\begin{align}
 \Delta\rho_n-c_n\rho_n\partial_x\theta_n
 -\rho_n|\nabla\theta_n|^2+(1-\rho_n^2)\rho_n&=0,
 \label{eq:4.23}\\
 \operatorname{div}(\rho_n^2\nabla\theta_n)
 +\frac{c_n}{2}\partial_x(\rho_n^2)&=0.
 \notag
\end{align}
Local convergence ensures nonvanishing in a fixed collar of every inner boundary circle. The phase therefore extends smoothly to these collars. Since $\chi$ is smooth and is zero on $[0,1]$, the products $\chi_{n,R}^2s_n$ and $\chi_{n,R}^2\theta_n$ extend smoothly by zero across the removed balls. To justify the test at infinity, first fix $n$ and $R$, and choose $S_0$ so large that all the moving balls $B(\mathbf{a}_n^\ell,2R)$ lie in $B_{S_0}$. For $S\ge S_0$, use the cutoff $\zeta_S$ from \eqref{eq:4.5}. Test \eqref{eq:4.23} by $\chi_{n,R}^2\zeta_S^2s_n$ and the divergence-free identity \eqref{eq:4.19} by $\chi_{n,R}^2\zeta_S^2\theta_n$. On each inner annulus subtract $\theta_{n,\ell}$ by \eqref{eq:4.21}. On $B_{2S}\setminus B_S$ subtract the average of $\theta_n$ on that annulus. Since $S\ge S_0$, we have $\chi_{n,R}=1$ on $B_{2S}\setminus B_S$. The outer cutoff term from the real equation satisfies
\begin{equation*}
 2\left|\int_{B_{2S}\setminus B_S}
 \zeta_Ss_n\nabla\zeta_S\cdot\nabla s_n\right|
 \le C\int_{B_{2S}\setminus B_S}
 (|\nabla s_n|^2+s_n^2)\longrightarrow0.
\end{equation*}
For the divergence equation, the scaled annular Poincar\'e inequality and $|\nabla\zeta_S|\le C/S$ give
\begin{align*}
 &2\left|\int_{B_{2S}\setminus B_S}\zeta_S
 \left(\theta_n-
 \frac{1}{|B_{2S}\setminus B_S|}
 \int_{B_{2S}\setminus B_S}\theta_n\right)
 \left\{\rho_n^2\nabla\theta_n
 +\frac{c_n}{2}(\rho_n^2-1)(1,0)\right\}
 \cdot\nabla\zeta_S\right|\\
 &\quad\le C
 \left(\int_{B_{2S}\setminus B_S}|\nabla\theta_n|^2\right)^{1/2}
 \left(\int_{B_{2S}\setminus B_S}
 (|\nabla\theta_n|^2+s_n^2)\right)^{1/2}
 \longrightarrow0.
\end{align*}
Here we used the polar energy formula following \eqref{eq:4.4} on the outer annulus. The terms supported on $A_{n,R}$ are independent of $S$ once $S\ge S_0$. The remaining bulk terms converge by dominated convergence, since their integrands are bounded by a constant times $|\nabla s_n|^2+|\nabla\theta_n|^2+s_n^2$, which is integrable on $\Omega_{n,R}$. This proves the passage to the limit $S\to\infty$ with $n$ and $R$ fixed.

Since $\nabla\rho_n=\nabla s_n$ and $(1-\rho_n^2)\rho_n=-2s_n-3s_n^2-s_n^3$, we deduce from the real equation that
\begin{align*}
 \int \Delta\rho_n\,\chi_{n,R}^2s_n
 &=-\int\chi_{n,R}^2|\nabla s_n|^2
   -2\int\chi_{n,R}s_n\nabla\chi_{n,R}\cdot\nabla s_n,\\
 -\int(1-\rho_n^2)\rho_ns_n\chi_{n,R}^2
 &=\int\chi_{n,R}^2(2s_n^2+3s_n^3+s_n^4).
\end{align*}
For the divergence equation, the coefficient of $\partial_x\theta_n$ is
\begin{equation*}
 \frac{c_n}{2}(\rho_n^2-1)
 =c_ns_n+\frac{c_n}{2}s_n^2.
\end{equation*}
Using these expansions in the tests of \eqref{eq:4.23} and \eqref{eq:4.19}, respectively, after letting $S\to\infty$ and subtracting the annular means by \eqref{eq:4.21}, we derive
\begin{align}
 &\int\chi_{n,R}^2\left(
 |\nabla s_n|^2+2s_n^2+c_ns_n\partial_x\theta_n\right)\notag\\
&\quad=-2\int\chi_{n,R}s_n\nabla\chi_{n,R}\cdot\nabla s_n
 -\int\chi_{n,R}^2\left\{
 c_ns_n^2\partial_x\theta_n
 +(1+s_n)s_n|\nabla\theta_n|^2+3s_n^3+s_n^4\right\},
 \label{eq:4.24}\\
 &\int\chi_{n,R}^2\left\{
 \rho_n^2|\nabla\theta_n|^2+c_ns_n\partial_x\theta_n
 +\frac{c_n}{2}s_n^2\partial_x\theta_n\right\}\notag\\
 &\quad=-2\sum_{\ell=1}^K
 \int_{\{R<|\mathbf{x}-\mathbf{a}_n^\ell|<2R\}}\chi_{n,R}
 (\theta_n-\theta_{n,\ell})
 \left\{\rho_n^2\nabla\theta_n
 +\frac{c_n}{2}(\rho_n^2-1)(1,0)\right\}
 \cdot\nabla\chi_{n,R}.
 \label{eq:4.25}
\end{align}
The right-hand side of \eqref{eq:4.25} is supported on $A_{n,R}$ because $\nabla\chi_{n,R}$ vanishes elsewhere. On this set the polar energy formula following \eqref{eq:4.4} and $|s_n|\le\eta_J<1/2$ show that $e(\psi_n)$ is bounded above and below by fixed positive multiples of $|\nabla s_n|^2+|\nabla\theta_n|^2+s_n^2$. It follows from Young's inequality and \eqref{eq:4.20} that
\begin{align}
 2\left|\int\chi_{n,R}s_n\nabla\chi_{n,R}\cdot\nabla s_n\right|
 &\le\int_{A_{n,R}}\chi_{n,R}^2|\nabla s_n|^2
 +\int_{A_{n,R}}s_n^2|\nabla\chi_{n,R}|^2
 \le C\int_{A_{n,R}}e(\psi_n).
 \label{eq:4.26}
\end{align}
Since $|s_n|\le\eta_J<1/2$ and $c_n\in J$, we also have
\begin{equation*}
 \left|\rho_n^2\nabla\theta_n
 +\frac{c_n}{2}(\rho_n^2-1)(1,0)\right|^2
 \le C_J\bigl(|\nabla\theta_n|^2+s_n^2\bigr).
\end{equation*}
Combining this inequality with \eqref{eq:4.22}, Hölder's inequality, and the disjointness of the annuli, we obtain
\begin{align}\label{eq:4.27}
&2\sum_{\ell=1}^K\left|
 \int_{\{R<|\mathbf{x}-\mathbf{a}_n^\ell|<2R\}}
 \chi_{n,R}(\theta_n-\theta_{n,\ell})
 \left\{\rho_n^2\nabla\theta_n
 +\frac{c_n}{2}(\rho_n^2-1)(1,0)\right\}
 \cdot\nabla\chi_{n,R}\right|\\
&\quad\le C\left(\int_{A_{n,R}}
 \bigl(|\nabla\theta_n|^2+s_n^2\bigr)\right)^{1/2}
 \left(\int_{A_{n,R}}|\nabla\theta_n|^2\right)^{1/2}
 \le C\int_{A_{n,R}}e(\psi_n).\nonumber
\end{align}
The nonlinear terms involving $s_n^2\partial_x\theta_n$, $(1+s_n)s_n|\nabla\theta_n|^2$, and $3s_n^3+s_n^4$ in \eqref{eq:4.24}--\eqref{eq:4.25} are controlled pointwise by
\begin{align*}
 |s_n^2\partial_x\theta_n|
 &\le\frac{|s_n|}{2}
       (s_n^2+|\partial_x\theta_n|^2),\\
 |(1+s_n)s_n|\nabla\theta_n|^2|
 &\le(1+\eta_J)|s_n||\nabla\theta_n|^2,\\
 |3s_n^3+s_n^4|&\le(3+\eta_J)|s_n|s_n^2.
\end{align*}
Since $c_n\in J$, integrating these bounds yields
\begin{align}
 &\int\chi_{n,R}^2\left|
 \frac{3c_n}{2}s_n^2\partial_x\theta_n
 +(1+s_n)s_n|\nabla\theta_n|^2+3s_n^3+s_n^4\right|\notag\\
 &\qquad\le C_J\|s_n\|_\infty
 \int\chi_{n,R}^2
 \left(|\nabla s_n|^2+|\nabla\theta_n|^2+s_n^2\right).
 \label{eq:4.28}
\end{align}
Adding \eqref{eq:4.24} and \eqref{eq:4.25}, and moving the $s_n^2\partial_x\theta_n$ term to the right, then using \eqref{eq:4.26}, \eqref{eq:4.27}, and \eqref{eq:4.28}, we infer
\begin{align}
 &\int\chi_{n,R}^2\left(
 |\nabla s_n|^2+\rho_n^2|\nabla\theta_n|^2+2s_n^2
 +2c_ns_n\partial_x\theta_n\right)\notag\\
 &\qquad\le
 C\|s_n\|_\infty\int\chi_{n,R}^2
 \left(|\nabla s_n|^2+|\nabla\theta_n|^2+s_n^2\right)
 +C\int_{A_{n,R}}e(\psi_n).
 \label{eq:4.29}
\end{align}

To bound the left-hand side of \eqref{eq:4.29}, replace $\rho_n^2|\nabla\theta_n|^2$ by $|\nabla\theta_n|^2$. The absolute value of the resulting error is at most $C\|s_n\|_\infty\int\chi_{n,R}^2|\nabla\theta_n|^2$. The quadratic form in $(\sqrt2s_n,\partial_x\theta_n)$ satisfies
\begin{equation}\label{eq:4.30}
 2s_n^2+|\partial_x\theta_n|^2
 +2c_ns_n\partial_x\theta_n
 \ge\left(1-\frac{\max J}{\sqrt2}\right)
 \left(2s_n^2+|\partial_x\theta_n|^2\right).
\end{equation}
Also $|\rho_n^2-1|\le(2+\|s_n\|_\infty)\|s_n\|_\infty$. Therefore, once $\|s_n\|_\infty$ is small enough, \eqref{eq:4.30} and the positive $|\nabla s_n|^2+|\partial_y\theta_n|^2$ terms show that the left-hand side of \eqref{eq:4.29} is at least
\begin{equation*}
 \frac12\left(1-\frac{\max J}{\sqrt2}\right)
 \int\chi_{n,R}^2
 (|\nabla s_n|^2+|\nabla\theta_n|^2+s_n^2).
\end{equation*}
First choose $R$ so large that \eqref{eq:4.16} makes $C\|s_n\|_\infty$ no larger than one quarter of $1-\max J/\sqrt2$ for all sufficiently large $n$. Then the first term on the right of \eqref{eq:4.29} is absorbed into the preceding lower bound, and
\begin{equation*}
 \frac14\left(1-\frac{\max J}{\sqrt2}\right)
 \int\chi_{n,R}^2
 (|\nabla s_n|^2+|\nabla\theta_n|^2+s_n^2)
 \le C\int_{A_{n,R}}e(\psi_n).
\end{equation*}
Since $\chi_{n,R}=1$ outside the balls of radius $2R$, the polar energy equivalence turns this inequality into
\begin{equation}\label{eq:4.31}
 \int_{\R^2\setminus\bigcup_\ell B(\mathbf{a}_n^\ell,2R)}
 e(\psi_n)\dd\mathbf{x}
 \le C_J\int_{A_{n,R}}e(\psi_n)\dd\mathbf{x}.
\end{equation}

We next take the limits in the required order. For fixed $R$, the balls and annuli about distinct centers are disjoint for all large $n$, and local smooth convergence at each center gives
\begin{equation}\label{eq:4.32}
 \lim_{n\to\infty}\int_{A_{n,R}}e(\psi_n)\dd\mathbf{x}
 =\sum_{\ell=1}^K
 \int_{B_{2R}\setminus B_R}e(\psi^\ell)\dd\mathbf{x}.
\end{equation}
The right-hand side of \eqref{eq:4.32} tends to zero as $R\to\infty$. Since the region outside the balls of radius $R$ is the disjoint union, up to boundaries, of $A_{n,R}$ and $\R^2\setminus\bigcup_\ell B(\mathbf{a}_n^\ell,2R)$, we conclude from \eqref{eq:4.31} that
\begin{equation}\label{eq:4.33}
 \int_{\{\mathbf{x}:\min_\ell|\mathbf{x}-\mathbf{a}_n^\ell|>R\}}
 e(\psi_n)\dd\mathbf{x}
 \le(1+C_J)\int_{A_{n,R}}e(\psi_n)\dd\mathbf{x}.
\end{equation}
Taking first $\limsup_{n\to\infty}$ in \eqref{eq:4.33}, using \eqref{eq:4.32}, and then taking $R\to\infty$ proves \eqref{eq:4.15}.
\end{proof}

\begin{proof}[Proof of Theorem~\ref{thm:1.3}]
The estimate proved in Theorem~\ref{thm:1.1} first converts the hypothesis \eqref{eq:1.7} into the energy bound
\begin{equation*}
 \sup_nE(\psi_n)<\infty.
\end{equation*}
We first pass to a subsequence on which $\ind(\psi_n)$ converges to its lower limit. Every subsequent extraction is taken from this subsequence. The derivative estimates quoted at the start of this section imply local precompactness, uniformly after arbitrary translations. If $\psi_n(\mathbf{z}_n+\cdot)\to v$ locally smoothly, then for every $R>0$,
\begin{equation*}
 \int_{B_R}e(v)\dd\mathbf{x}
 =\lim_{n\to\infty}\int_{B(\mathbf{z}_n,R)}e(\psi_n)\dd\mathbf{x}
 \le\sup_nE(\psi_n).
\end{equation*}
Letting $R\to\infty$ shows that $E(v)<\infty$. A constant solution of \eqref{eq:1.2} is either $0$ or has modulus one. The constant $0$ has infinite energy on $\R^2$. Hence every constant translated limit has modulus one.

Starting with no centers, we select nonconstant translated limits whose centers separate from all previously chosen centers. Suppose that centers $\mathbf{a}_n^1,\ldots,\mathbf{a}_n^m$ have been chosen. If every sequence $\mathbf{z}_n$ satisfying
\begin{equation}\label{eq:4.34}
 \min_{1\le\ell\le m}|\mathbf{z}_n-\mathbf{a}_n^\ell|\longrightarrow\infty
\end{equation}
has only constants of modulus one as locally smooth subsequential limits, we stop (with the convention in Lemma~\ref{lem:4.3} when $m=0$). Otherwise choose $\mathbf{z}_n$ satisfying \eqref{eq:4.34} such that a subsequence of $\psi_n(\mathbf{z}_n+\cdot)$ has a nonconstant local limit, and take $\mathbf{a}_n^{m+1}=\mathbf{z}_n$. After extraction there are phases $\alpha_n^{m+1}\in\mathbb S^1$ and a normalized nonconstant finite-energy wave $\psi^{m+1}$ at speed $c$ such that
\begin{equation*}
 \overline{\alpha_n^{m+1}}\,
 \psi_n(\mathbf{a}_n^{m+1}+\cdot)\longrightarrow\psi^{m+1}
 \quad\text{locally together with all derivatives}.
\end{equation*}
The equation passes to the limit because $c_n\to c$, and the energy argument above shows that the limit has finite energy. We choose the constant phase so that $\psi^{m+1}$ tends to $1$ at infinity. Equation \eqref{eq:4.34} separates the new center from every previously chosen center. At each repetition we pass to a further subsequence. All convergences and separations obtained earlier remain valid. The energy estimate in the next paragraph shows that only finitely many repetitions are possible, so the final subsequence retains every selected profile.

For each of the first $m$ profiles choose $\rho_\ell>0$ such that
\begin{equation*}
 \int_{B_{\rho_\ell}}e(\psi^\ell)\ge e_J/2.
\end{equation*}
This is possible by Lemma~\ref{lem:4.2}. For large $n$ the corresponding translated balls are disjoint, and local convergence yields
\begin{equation*}
 \frac{m e_J}{2}
 \le\liminf_{n\to\infty}
 \sum_{\ell=1}^m\int_{B(\mathbf{a}_n^\ell,\rho_\ell)}e(\psi_n)
 \le \sup_nE(\psi_n).
\end{equation*}
Thus the induction stops after at most $2\sup_nE(\psi_n)/e_J$ profiles. Let $K$ be the final number. The construction establishes \eqref{eq:1.8} and \eqref{eq:1.9}.

At termination, the stopping condition is exactly the hypothesis of Lemma~\ref{lem:4.3}. Therefore
\begin{equation}\label{eq:4.35}
 \lim_{R\to\infty}\limsup_n
 \int_{\R^2\setminus\bigcup_\ell B(\mathbf{a}_n^\ell,R)}
 e(\psi_n)=0.
\end{equation}

If $K=0$, take $R_n=n$. It follows from \eqref{eq:4.35} that $E(\psi_n)\to0$. Since each energy component is nonnegative, \eqref{eq:3.2} yields \eqref{eq:1.11}-- \eqref{eq:1.14}. Equation \eqref{eq:1.15} is empty. Assume $K\ge1$.

To use the local convergence and \eqref{eq:4.35} simultaneously, we choose one diagonal sequence of radii. For each integer $m\ge1$, choose $r_m\ge m$, increasing in $m$, such that, for every $1\le\ell\le K$,
\begin{equation}\label{eq:4.36}
 \int_{\R^2\setminus B_{r_m}}
 \left(|\nabla\psi^\ell|^2+|\partial_y\psi^\ell|^2
 +(1-|\psi^\ell|^2)^2\right)\dd\mathbf{x}<\frac1m,
\end{equation}
and, by \eqref{eq:4.35},
\begin{equation}\label{eq:4.37}
 \limsup_{n\to\infty}
 \int_{\R^2\setminus\bigcup_{\ell=1}^K B(\mathbf{a}_n^\ell,r_m)}
 e(\psi_n)\dd\mathbf{x}<\frac1m.
\end{equation}
The three densities used in \eqref{eq:4.36} are
\begin{equation}\label{eq:4.38}
 |\nabla\psi^\ell|^2,\qquad |\partial_y\psi^\ell|^2,
 \qquad (1-|\psi^\ell|^2)^2.
\end{equation}
By local convergence on the finitely many balls, choose strictly increasing integers $N_m$ such that, for every $n\ge N_m$ and $1\le\ell\le K$,
\begin{align*}
 \left|\int_{B(\mathbf{a}_n^\ell,r_m)}|\nabla\psi_n|^2
 -\int_{B_{r_m}}|\nabla\psi^\ell|^2\right|&<\frac1m,\\
 \left|\int_{B(\mathbf{a}_n^\ell,r_m)}|\partial_y\psi_n|^2
 -\int_{B_{r_m}}|\partial_y\psi^\ell|^2\right|&<\frac1m,\\
 \left|\int_{B(\mathbf{a}_n^\ell,r_m)}(1-|\psi_n|^2)^2
 -\int_{B_{r_m}}(1-|\psi^\ell|^2)^2\right|&<\frac1m.
\end{align*}
Increase $N_m$, if necessary, so that the exterior energy in \eqref{eq:4.37} is also less than $2/m$ for every $n\ge N_m$. If $K\ge2$, increase $N_m$ so that
\begin{equation}\label{eq:4.39}
 \frac{r_m}{\min_{k\ne\ell}|\mathbf{a}_n^k-\mathbf{a}_n^\ell|}<\frac1m
 \qquad(n\ge N_m).
\end{equation}
After increasing $N_m$ once more, the balls $B(\mathbf{a}_n^\ell,r_m)$ are pairwise disjoint for every $n\ge N_m$. This also prevents overlap when the local integrals are added. Define
\begin{equation*}
 R_n:=r_m\qquad(N_m\le n<N_{m+1}),
\end{equation*}
and choose the finitely many earlier values to be positive. Then
\begin{equation}\label{eq:4.40}
 R_n\longrightarrow\infty.
\end{equation}
When $K\ge2$, we also infer from \eqref{eq:4.39} that
\begin{equation}\label{eq:4.41}
 \frac{R_n}{\min_{k\ne\ell}|\mathbf{a}_n^k-\mathbf{a}_n^\ell|}\longrightarrow0.
\end{equation}
Equations \eqref{eq:4.36}, \eqref{eq:4.37}, and \eqref{eq:4.40}, together with \eqref{eq:4.41} when $K\ge2$, imply \eqref{eq:1.10}. We spell out how the same diagonal choice splits the three integrals. Let $g_n$ be, in turn, $e(\psi_n)$, $|\partial_y\psi_n|^2$, or $(1-|\psi_n|^2)^2$, and let $g^\ell$ denote the corresponding density of $\psi^\ell$. Disjointness of the balls gives the exact decomposition
\begin{align*}
 \int_{\R^2}g_n\dd\mathbf{x}
 -\sum_{\ell=1}^K\int_{\R^2}g^\ell\dd\mathbf{x}
 &=\sum_{\ell=1}^K\left\{
 \int_{B(\mathbf{a}_n^\ell,R_n)}g_n\dd\mathbf{x}
 -\int_{B_{R_n}}g^\ell\dd\mathbf{x}\right\}\\
 &\quad+\int_{\R^2\setminus\bigcup_\ell
 B(\mathbf{a}_n^\ell,R_n)}g_n\dd\mathbf{x}
 -\sum_{\ell=1}^K\int_{\R^2\setminus B_{R_n}}g^\ell\dd\mathbf{x}.
\end{align*}
If $N_m\le n<N_{m+1}$, the absolute value of the first line is at most $C_K/m$ by the definition of $N_m$ and the formula for $e$. Equations \eqref{eq:4.36} and \eqref{eq:4.37} bound the absolute value of the second line by $C_K/m$. For the first two densities use $|\partial_y\psi|^2\le|\nabla\psi|^2\le2e(\psi)$, and for the third use $(1-|\psi|^2)^2\le4e(\psi)$. Hence the displayed difference tends to zero. The three choices of $g_n$ prove \eqref{eq:1.11}, \eqref{eq:1.12} using the first identity in \eqref{eq:3.2}, and \eqref{eq:1.13}, respectively.

For momentum, the second identity in \eqref{eq:3.2} is
\begin{equation}\label{eq:4.42}
 c_nP(\psi_n)=\frac12\int_{\R^2}(1-|\psi_n|^2)^2\dd\mathbf{x}.
\end{equation}
Since $c_n\to c>0$, we derive directly from \eqref{eq:1.13} and \eqref{eq:4.42} that
\begin{equation*}
 P(\psi_n)=\frac1{2c_n}\int_{\R^2}(1-|\psi_n|^2)^2\dd\mathbf{x}
 \longrightarrow\sum_{\ell=1}^K\frac1{2c}
 \int_{\R^2}(1-|\psi^\ell|^2)^2\dd\mathbf{x}
 =\sum_{\ell=1}^KP(\psi^\ell),
\end{equation*}
which is \eqref{eq:1.14}.

It remains to prove the index inequality. For each $1\le\ell\le K$, choose any finite-dimensional real vector space on which $Q_{\psi^\ell,c}$ is negative definite, and choose a real basis so that
\begin{equation}\label{eq:4.43}
 Y^\ell=\operatorname{span}
 \{\zeta_{\ell,1},\ldots,\zeta_{\ell,d_\ell}\}
 \subset C_c^\infty(\R^2,\C).
\end{equation}
If $\sum_\ell d_\ell=0$, the desired bound is immediate. Hence assume $\sum_\ell d_\ell>0$ and omit zero-dimensional spaces from the maxima in the next formulas. Compactness of the Euclidean unit spheres in the finitely many spaces $Y^\ell$ supplies $\delta>0$ such that
\begin{equation}\label{eq:4.44}
 Q_{\psi^\ell,c}\!\left(\sum_{p=1}^{d_\ell}t_{\ell,p}
 \zeta_{\ell,p}\right)
 \le-2\delta\sum_{p=1}^{d_\ell}t_{\ell,p}^2
 \quad(1\le\ell\le K).
\end{equation}
Here all coefficients $t_{\ell,p}$ are real. Define
\begin{equation*}
 \zeta_{n,\ell,p}(\mathbf{x})
 :=\alpha_n^\ell\zeta_{\ell,p}(\mathbf{x}-\mathbf{a}_n^\ell).
\end{equation*}
Tests with distinct profile labels $\ell$ have disjoint supports for large $n$ by \eqref{eq:1.8}. In the term with profile label $\ell$, we change variables by $\mathbf{x}=\mathbf{a}_n^\ell+\mathbf{z}$ and multiply both the solution and the test by $\overline{\alpha_n^\ell}$. The form \eqref{eq:1.4} is unchanged by this constant rotation. By \eqref{eq:1.9}, $c_n\to c$, and the coefficient formula \eqref{eq:1.4}, the coefficients of the quadratic form converge uniformly on a fixed compact set in the $\mathbf{z}$ coordinates containing the supports of $\zeta_{\ell,1},\ldots,\zeta_{\ell,d_\ell}$. Since each $Y^\ell$ is finite-dimensional, we obtain
\begin{equation}\label{eq:4.45}
 \max_{1\le\ell\le K}
 \sup_{\substack{t_1,\ldots,t_{d_\ell}\in\R\\
                   \sum_{p=1}^{d_\ell}t_p^2=1}}
 \left|Q_{\psi_n,c_n}\!\left(\sum_{p=1}^{d_\ell}
 t_p\zeta_{n,\ell,p}\right)
 -Q_{\psi^\ell,c}\!\left(\sum_{p=1}^{d_\ell}
 t_p\zeta_{\ell,p}\right)\right|
 \longrightarrow0.
\end{equation}
Mixed terms with distinct profile labels $\ell$ vanish because their supports are disjoint. We therefore conclude from equations \eqref{eq:4.44} and \eqref{eq:4.45} that, for all large $n$,
\begin{equation*}
 Q_{\psi_n,c_n}\!\left(
 \sum_{\ell=1}^K\sum_{p=1}^{d_\ell}
 t_{\ell,p}\zeta_{n,\ell,p}\right)
 \le-\delta\sum_{\ell=1}^K\sum_{p=1}^{d_\ell}t_{\ell,p}^2.
\end{equation*}
Consequently, \eqref{eq:1.5} implies
\begin{equation*}
 \sum_{\ell=1}^K d_\ell\le\ind(\psi_n)
 \quad\text{for all sufficiently large }n,
\end{equation*}
and hence
\begin{equation}\label{eq:4.46}
 \sum_{\ell=1}^K d_\ell
 \le\liminf_{n\to\infty}\ind(\psi_n).
\end{equation}
If every profile index is finite, choose the spaces in \eqref{eq:4.43} with $d_\ell=\ind(\psi^\ell)$. Equation \eqref{eq:4.46} yields \eqref{eq:1.15}. If some profile index is $+\infty$, its space $Y^\ell$ can be chosen with arbitrarily large finite dimension. Then \eqref{eq:4.46} proves the same inequality in the extended nonnegative integers.
\end{proof}

\subsection{Minimization among waves of index at most one}

\begin{proof}[Proof of Theorem~\ref{thm:1.4}]
Theorem~\ref{thm:1.2} makes the set in \eqref{eq:1.16} nonempty. The calculation \eqref{eq:3.31}--\eqref{eq:3.33} also proves the upper bound $\beta_1(c)\le A_J$ in \eqref{eq:1.17}. We first prove that the action cannot approach zero along nonconstant waves of index at most one. Fix a compact interval $J\subset(0,\sqrt2)$ and suppose that no $b_J>0$ exists. Then there are $c_n\in J$ and nonconstant finite-energy waves $v_n$ at speed $c_n$ such that
\begin{equation}\label{eq:4.47}
 I_{c_n}(v_n)\longrightarrow0,
 \qquad \ind(v_n)\le1.
\end{equation}
After extraction, $c_n\to c\in J$. By Theorem~\ref{thm:1.1} and \eqref{eq:4.47}, we have
\begin{equation}\label{eq:4.48}
 \sup_nE(v_n)\le C_J\sup_n\bigl(I_{c_n}(v_n)+1\bigr)<\infty.
\end{equation}
The uniform modulus gap \eqref{eq:4.14} provides points $\mathbf{z}_n\in\R^2$ such that
\begin{equation}\label{eq:4.49}
 \bigl|1-|v_n(\mathbf{z}_n)|\bigr|>\eta_J.
\end{equation}
Translate and rotate by setting
\begin{equation*}
 w_n(\mathbf{x}):=
 \begin{cases}
 \dfrac{\overline{v_n(\mathbf{z}_n)}}{|v_n(\mathbf{z}_n)|}\,
 v_n(\mathbf{z}_n+\mathbf{x}),
   &v_n(\mathbf{z}_n)\ne0,\\[5pt]
 v_n(\mathbf{z}_n+\mathbf{x}),&v_n(\mathbf{z}_n)=0.
 \end{cases}
\end{equation*}
By the derivative estimates quoted at the start of this section and a subsequence, we obtain a solution $v$ at speed $c$ such that
\begin{equation*}
 w_n\longrightarrow v
 \quad\text{locally together with all derivatives}.
\end{equation*}
Fatou's lemma, \eqref{eq:4.48}, and \eqref{eq:4.49} yield
\begin{equation}\label{eq:4.50}
 E(v)\le\liminf_{n\to\infty}E(w_n)<\infty,
 \qquad \bigl|1-|v(0)|\bigr|\ge\eta_J.
\end{equation}
A constant solution $v\equiv q$ satisfies $(1-|q|^2)q=0$, so either $q=0$ or $|q|=1$. The constant $q=0$ has infinite energy, while $|q|=1$ contradicts the modulus gap in \eqref{eq:4.50}. Thus $v$ is nonconstant.

For every $R>0$, it follows from \eqref{eq:3.2}, \eqref{eq:4.47}, and local convergence that
\begin{equation*}
 \int_{B_R}|\partial_yv|^2\dd\mathbf{x}
 =\lim_{n\to\infty}\int_{B_R}|\partial_yw_n|^2\dd\mathbf{x}
 \le\lim_{n\to\infty}I_{c_n}(v_n)=0.
\end{equation*}
Hence $\partial_yv=0$ and $v(x,y)=q(x)$ for a smooth nonconstant function $q$. Since $q$ is nonconstant, there are $a<b$ for which
\begin{equation*}
 \int_a^b\left\{\frac12|q'|^2
 +\frac14(1-|q|^2)^2\right\}\dd x>0.
\end{equation*}
If this integral vanished on every bounded interval, then $q'=0$ and $|q|=1$ everywhere. For every $L>0$, we now integrate over the rectangle $(a,b)\times(-L,L)$ to obtain
\begin{equation*}
 \begin{aligned}
 E(v)&\ge\int_{-L}^{L}\int_a^b
 \left\{\frac12|q'(x)|^2+\frac14(1-|q(x)|^2)^2\right\}\dd x\dd y\\
 &=2L\int_a^b\left\{\frac12|q'|^2
 +\frac14(1-|q|^2)^2\right\}\dd x\longrightarrow\infty.
 \end{aligned}
\end{equation*}
This contradicts \eqref{eq:4.50} and proves the lower bound in \eqref{eq:1.17}.

We next prove attainment. Fix $c$ and let $v_n$ be a minimizing sequence in \eqref{eq:1.16}. Explicitly,
\begin{equation}\label{eq:4.51}
 I_c(v_n)\longrightarrow\beta_1(c),
 \qquad\ind(v_n)\le1.
\end{equation}
Using the normalization convention fixed after \eqref{eq:1.3}, choose the representative of each $v_n$ that tends to $1$ at infinity. Theorem~\ref{thm:1.1} bounds $E(v_n)$, so Theorem~\ref{thm:1.3} applies. If its profile number were $K=0$, \eqref{eq:1.12} would imply $I_c(v_n)\to0$, contrary to the positive lower bound already proved in \eqref{eq:1.17}. Thus $K\ge1$. We infer from equations \eqref{eq:1.15} and \eqref{eq:4.51} that
\begin{equation*}
 \sum_{\ell=1}^K\ind(\psi^\ell)\le1,
\end{equation*}
so every nonconstant profile $\psi^\ell$ is admissible in \eqref{eq:1.16}. Exact action splitting implies
\begin{equation}\label{eq:4.52}
 \beta_1(c)=\sum_{\ell=1}^K I_c(\psi^\ell)
 \ge K\beta_1(c).
\end{equation}
Since $\beta_1(c)>0$ and $K\ge1$, equation \eqref{eq:4.52} forces $K=1$. Denote the sole profile by $\psi_c$. Then
\begin{equation}\label{eq:4.53}
 I_c(\psi_c)=\beta_1(c),\qquad \ind(\psi_c)\le1.
\end{equation}
For $c\in J$, we conclude from equations \eqref{eq:1.6}, \eqref{eq:1.17}, and \eqref{eq:4.53} that
\begin{equation*}
 E(\psi_c)\le C_J\{I_c(\psi_c)+\ind(\psi_c)\}
 \le C_J(A_J+1),
\end{equation*}
which is \eqref{eq:1.18}.

We also prove the asserted precompactness of every minimizing sequence. Set $w_n=\overline{\alpha_n}v_n(\mathbf{a}_n+\,\cdot)$, using the translations and phases from Theorem~\ref{thm:1.3}. The uniform energy bound and local smooth convergence imply the following weak $L^2$ convergences. The arrow $\rightharpoonup$ denotes weak convergence:
\begin{equation}\label{eq:4.54}
 \nabla w_n\rightharpoonup\nabla\psi_c
 \quad\text{and}\quad
 1-|w_n|^2\rightharpoonup1-|\psi_c|^2
 \qquad\text{weakly in }L^2(\R^2).
\end{equation}
The $L^2(\R^2)$ bounds on $\nabla w_n$ and $1-|w_n|^2$ allow weakly convergent subsequences. Local smooth convergence identifies their weak limits as $\nabla\psi_c$ and $1-|\psi_c|^2$ against compactly supported smooth tests. Density then extends this identification to all $L^2$ tests.

For $K=1$, we obtain from \eqref{eq:1.13} and \eqref{eq:1.11} that
\begin{equation}\label{eq:4.55}
 \|1-|w_n|^2\|_2^2\longrightarrow\|1-|\psi_c|^2\|_2^2,
 \qquad E(w_n)\longrightarrow E(\psi_c).
\end{equation}
By the definition of $E$,
\begin{equation*}
 \|\nabla w_n\|_2^2
 =2E(w_n)-\frac12\|1-|w_n|^2\|_2^2,
\end{equation*}
and the same identity holds for $\psi_c$. Thus \eqref{eq:4.55} implies
\begin{equation}\label{eq:4.56}
 \|\nabla w_n\|_2\longrightarrow\|\nabla\psi_c\|_2,
 \qquad \|1-|w_n|^2\|_2\longrightarrow
 \|1-|\psi_c|^2\|_2.
\end{equation}
Weak convergence and \eqref{eq:4.56} yield
\begin{align}
 \|\nabla w_n-\nabla\psi_c\|_2^2
 &=\|\nabla w_n\|_2^2+\|\nabla\psi_c\|_2^2
 -2\Rea\int\nabla w_n\cdot\overline{\nabla\psi_c}\longrightarrow0,
 \notag\\
 \||w_n|^2-|\psi_c|^2\|_2^2
 &=\|1-|w_n|^2\|_2^2+\|1-|\psi_c|^2\|_2^2-2\int(1-|w_n|^2)(1-|\psi_c|^2)\longrightarrow0.
 \label{eq:4.57}
\end{align}
Equations \eqref{eq:4.54}-- \eqref{eq:4.57} prove $\|\nabla w_n-\nabla\psi_c\|_2\to0$ and $\||w_n|^2-|\psi_c|^2\|_2\to0$, as asserted in \eqref{eq:1.19}.

Finally, let $c_n\to c$ and pass to a subsequence on which $\beta_1(c_n)$ converges to its lower limit. Choose a minimizer $v_n$ at speed $c_n$, normalize it by a constant phase, and choose a compact interval $J'\subset(0,\sqrt2)$ containing $c$ in its interior. After discarding finitely many terms, $c_n\in J'$ for every $n$. The local bounds \eqref{eq:1.17} and \eqref{eq:1.6}, used with $J'$, show that the hypotheses of Theorem~\ref{thm:1.3} hold. The profile family is nonempty: if $K=0$, the action splitting \eqref{eq:1.12} would imply $\beta_1(c_n)=I_{c_n}(v_n)\to0$, contradicting $\beta_1(c_n)\ge b_{J'}$. Moreover, \eqref{eq:1.15} makes every profile admissible in \eqref{eq:1.16}. We therefore infer from the action splitting that
\begin{equation*}
 \liminf_{n\to\infty}\beta_1(c_n)
 =\sum_{\ell=1}^K I_c(\psi^\ell)
 \ge K\beta_1(c)\ge\beta_1(c).
\end{equation*}
This proves lower semicontinuity and completes the theorem.
\end{proof}

\section*{Declarations}
\noindent\textbf{Funding.} Changfeng Gui and Shanfa Lai are supported by NSFC Key Program (Grant No.12531010), University of Macau research grants CPG2024-00016- FST, CPG2025-00032-FST, CPG2026-00027-FST, SRG2023-00011-FST, MYRGGRG2023-00139-FST-UMDF, UMDF Professorial Fellowship of Mathematics, Macao SAR FDCT 0003/2023/RIA1 and Macao SAR FDCT 0024/2023/RIB1. G. Qin was supported by National Key R\&D Program of China (Grant 2025YFA1018400) and NNSF of China (Grant 12471190). Juncheng Wei was supported by National Key R\&D Program of China (Grant No.2022YFA1005602), and Hong Kong General Research Fund ``New frontiers in singular limits of nonlinear partial differential equation". The authors acknowledge the use of OpenAI to assist with language editing and manuscript preparation. The authors take full responsibility for the content of the manuscript.

\noindent\textbf{Author contributions.} All authors contributed equally.

\noindent\textbf{Conflict of interest.} On behalf of all authors, the corresponding author states that there is no conflict of interest.

\noindent\textbf{Data availability.} Data availability is not applicable to this article, as no datasets were generated or analysed during the current study.


\begin{thebibliography}{99}

\bibitem{AoHuangLiuWei}
W.~Ao, Y.~Huang, Y.~Liu, and J.~Wei, \emph{Generalized Adler--Moser polynomials and multiple vortex rings for the Gross--Pitaevskii equation}, SIAM J. Math. Anal. \textbf{53} (2021), no.~6, 6959--6992.

\bibitem{BellazziniRuiz}
J.~Bellazzini and D.~Ruiz, \emph{Finite energy traveling waves for the Gross--Pitaevskii equation in the subsonic regime}, Amer. J. Math. \textbf{145} (2023), no.~1, 109--149.

\bibitem{Berloff}
N.~G.~Berloff, \emph{Quantised vortices, travelling coherent structures and superfluid turbulence}, in \emph{Stationary and Time Dependent Gross--Pitaevskii Equations}, Contemp. Math. \textbf{473}, Amer. Math. Soc., Providence, RI, 2008, 27--54.

\bibitem{BethuelGravejatSautSurvey}
F.~B\'ethuel, P.~Gravejat, and J.-C.~Saut, \emph{Existence and properties of travelling waves for the Gross--Pitaevskii equation}, in \emph{Stationary and Time Dependent Gross--Pitaevskii Equations}, Contemp. Math. \textbf{473}, Amer. Math. Soc., Providence, RI, 2008, 55--104.

\bibitem{BethuelGravejatSautKP}
F.~B\'ethuel, P.~Gravejat, and J.-C.~Saut, \emph{On the KP-I transonic limit of two-dimensional Gross--Pitaevskii travelling waves}, Dyn. Partial Differ. Equ. \textbf{5} (2008), no.~3, 241--280.

\bibitem{BethuelGravejatSaut}
F.~B\'ethuel, P.~Gravejat, and J.-C.~Saut, \emph{Travelling waves for the Gross--Pitaevskii equation. II}, Comm. Math. Phys. \textbf{285} (2009), 567--651.

\bibitem{BethuelOrlandiSmets}
F.~B\'ethuel, G.~Orlandi, and D.~Smets, \emph{Vortex rings for the Gross--Pitaevskii equation}, J. Eur. Math. Soc. (JEMS) \textbf{6} (2004), no.~1, 17--94.

\bibitem{BethuelSaut}
F.~B\'ethuel and J.-C.~Saut, \emph{Travelling waves for the Gross--Pitaevskii equation. I}, Ann. Inst. H. Poincar\'e Phys. Th\'eor. \textbf{70} (1999), 147--238.

\bibitem{BurtonVortexPairs}
G.~R.~Burton, \emph{Steady symmetric vortex pairs and rearrangements}, Proc. Roy. Soc. Edinburgh Sect. A \textbf{108} (1988), no.~3--4, 269--290.

\bibitem{CaoLaiQin}
D.~Cao, S.~Lai, and G.~Qin, \emph{Slow traveling-wave solutions for the generalized surface quasi-geostrophic equation}, J. Funct. Anal. \textbf{287} (2024), no.~8, Paper No.~110570, 59~pp.

\bibitem{CaoLaiZhan}
D.~Cao, S.~Lai, and W.~Zhan, \emph{Traveling vortex pairs for $2D$ incompressible Euler equations}, Calc. Var. Partial Differential Equations \textbf{60} (2021), no.~5, Paper No.~190, 16~pp.

\bibitem{ChironHigherDimensional}
D.~Chiron, \emph{Travelling waves for the Gross--Pitaevskii equation in dimension larger than two}, Nonlinear Anal. \textbf{58} (2004), no.~1--2, 175--204.

\bibitem{ChironHelices}
D.~Chiron, \emph{Vortex helices for the Gross--Pitaevskii equation}, J. Math. Pures Appl. (9) \textbf{84} (2005), no.~11, 1555--1647.

\bibitem{ChironRarefaction}
D.~Chiron, \emph{Smooth branch of rarefaction pulses for the nonlinear Schr\"odinger equation and the Euler--Korteweg system in $2d$}, Ann. Henri Lebesgue \textbf{6} (2023), 767--845.

\bibitem{ChironMaris}
D.~Chiron and M.~Mari\c{s}, \emph{Traveling waves for nonlinear Schr\"odinger equations with nonzero conditions at infinity}, Arch. Ration. Mech. Anal. \textbf{226} (2017), 143--242.

\bibitem{ChironPacherieBranch}
D.~Chiron and E.~Pacherie, \emph{Smooth branch of travelling waves for the Gross--Pitaevskii equation in $\mathbb R^2$ for small speed}, Ann. Sc. Norm. Super. Pisa Cl. Sci. (5) \textbf{22} (2021), 1937--2038.

\bibitem{ChironPacherieCoercivity}
D.~Chiron and E.~Pacherie, \emph{Coercivity for travelling waves in the Gross--Pitaevskii equation in $\mathbb R^2$ for small speed}, Publ. Mat. \textbf{67} (2023), no.~1, 277--410.

\bibitem{ChironPacherieUniqueness}
D.~Chiron and E.~Pacherie, \emph{A uniqueness result for the two-vortex traveling wave in the nonlinear Schr\"odinger equation}, Anal. PDE \textbf{16} (2023), no.~9, 2173--2224.

\bibitem{ChironScheid}
D.~Chiron and C.~Scheid, \emph{Multiple branches of travelling waves for the Gross--Pitaevskii equation}, Nonlinearity \textbf{31} (2018), no.~6, 2809--2853.

\bibitem{DavilaDelPinoMedinaRodiac}
J.~D\'avila, M.~del Pino, M.~Medina, and R.~Rodiac, \emph{Interacting helical traveling waves for the Gross--Pitaevskii equation}, Ann. Inst. H. Poincar\'e C Anal. Non Lin\'eaire \textbf{39} (2022), no.~6, 1319--1367.

\bibitem{DeLaireGravejatSmets}
A.~de Laire, P.~Gravejat, and D.~Smets, \emph{Construction of minimizing travelling waves for the Gross--Pitaevskii equation on $\mathbb R\times\mathbb T$}, Tunisian J. Math. \textbf{6} (2024), no.~1, 157--188.

\bibitem{FangGhoussoub}
G.~Fang and N.~Ghoussoub, \emph{Second-order information on Palais--Smale sequences in the mountain pass theorem}, Manuscripta Math. \textbf{75} (1992), no.~1, 81--95.

\bibitem{GravejatSupersonic}
P.~Gravejat, \emph{A non-existence result for supersonic travelling waves in the Gross--Pitaevskii equation}, Comm. Math. Phys. \textbf{243} (2003), no.~1, 93--103.

\bibitem{GravejatDecay}
P.~Gravejat, \emph{Decay for travelling waves in the Gross--Pitaevskii equation}, Ann. Inst. H. Poincar\'e C Anal. Non Lin\'eaire \textbf{21} (2004), 591--637.

\bibitem{GerardZhang}
P.~G\'erard and Z.~Zhang, \emph{Orbital stability of traveling waves for the one-dimensional Gross--Pitaevskii equation}, J. Math. Pures Appl. (9) \textbf{91} (2009), no.~2, 178--210.

\bibitem{GravejatAsymptotics}
P.~Gravejat, \emph{Asymptotics for the travelling waves in the Gross--Pitaevskii equation}, Asymptot. Anal. \textbf{45} (2005), no.~3--4, 227--299.

\bibitem{GravejatFirstOrder}
P.~Gravejat, \emph{First order asymptotics for the travelling waves in the Gross--Pitaevskii equation}, Adv. Differential Equations \textbf{11} (2006), no.~3, 259--280.

\bibitem{GravejatPacherieSmets}
P.~Gravejat, E.~Pacherie, and D.~Smets, \emph{On the stability of the Ginzburg--Landau vortex}, Proc. Lond. Math. Soc. (3) \textbf{125} (2022), no.~5, 1015--1065.

\bibitem{GravejatSonic}
P.~Gravejat, \emph{Limit at infinity and nonexistence results for sonic travelling waves in the Gross--Pitaevskii equation}, Differential Integral Equations \textbf{17} (2004), no.~11--12, 1213--1232.

\bibitem{GravejatSmets}
P.~Gravejat and D.~Smets, \emph{Asymptotic stability of the black soliton for the Gross--Pitaevskii equation}, Proc. Lond. Math. Soc. (3) \textbf{111} (2015), no.~2, 305--353.

\bibitem{Gross}
E.~P.~Gross, \emph{Structure of a quantized vortex in boson systems}, Nuovo Cimento (10) \textbf{20} (1961), no.~3, 454--477.

\bibitem{GuiHamiltonian}
C.~Gui, \emph{Hamiltonian identities for elliptic partial differential equations}, J. Funct. Anal. \textbf{254} (2008), no.~4, 904--933.

\bibitem{JonesPuttermanRoberts}
C.~A.~Jones, S.~J.~Putterman, and P.~H.~Roberts, \emph{Motions in a Bose condensate V. Stability of solitary wave solutions of nonlinear Schr\"odinger equations in two and three dimensions}, J. Phys. A \textbf{19} (1986), 2991--3011.

\bibitem{JonesRoberts}
C.~A.~Jones and P.~H.~Roberts, \emph{Motions in a Bose condensate. IV. Axisymmetric solitary waves}, J. Phys. A \textbf{15} (1982), 2599--2619.

\bibitem{KochLiao}
H.~Koch and X.~Liao, \emph{Conserved energies for the one dimensional Gross--Pitaevskii equation}, Adv. Math. \textbf{377} (2021), Paper No.~107467, 83~pp.

\bibitem{KochLiaoLowRegularity}
H.~Koch and X.~Liao, \emph{Conserved energies for the one dimensional Gross--Pitaevskii equation: low regularity case}, Adv. Math. \textbf{420} (2023), Paper No.~108996, 61~pp.

\bibitem{LinXin}
F.-H.~Lin and J.~X.~Xin, \emph{On the incompressible fluid limit and the vortex motion law of the nonlinear Schr\"odinger equation}, Comm. Math. Phys. \textbf{200} (1999), no.~2, 249--274.

\bibitem{LinWeiObstacle}
F.-H.~Lin and J.~Wei, \emph{Superfluids passing an obstacle and vortex nucleation}, Discrete Contin. Dyn. Syst. \textbf{39} (2019), no.~12, 6801--6824.

\bibitem{LionsConcentration}
P.-L.~Lions, \emph{The concentration-compactness principle in the calculus of variations. The locally compact case. I}, Ann. Inst. H. Poincar\'e C Anal. Non Lin\'eaire \textbf{1} (1984), no.~2, 109--145.

\bibitem{LiuWei}
Y.~Liu and J.~Wei, \emph{Multivortex traveling waves for the Gross--Pitaevskii equation and the Adler--Moser polynomials}, SIAM J. Math. Anal. \textbf{52} (2020), no.~4, 3546--3579.

\bibitem{LiuWangWeiYang}
Y.~Liu, Z.~Wang, J.~Wei, and W.~Yang, \emph{From KP-I lump solution to travelling waves of Gross--Pitaevskii equation}, J. Math. Pures Appl. (9) \textbf{205} (2026), Paper No.~103801.

\bibitem{Maris}
M.~Mari\c{s}, \emph{Traveling waves for nonlinear Schr\"odinger equations with nonzero conditions at infinity}, Ann. of Math. (2) \textbf{178} (2013), 107--182.

\bibitem{MartinezSanchezRuiz}
F.~J.~Mart{\'\i}nez S\'anchez and D.~Ruiz, \emph{Existence and nonexistence of traveling waves for the Gross--Pitaevskii equation in tori}, Math. Eng. \textbf{5} (2023), no.~1, Paper No.~011, 14~pp.

\bibitem{PacherieSurvey}
E.~Pacherie, \emph{A uniqueness result for travelling waves in the Gross--Pitaevskii equation}, S\'eminaire Laurent-Schwartz---EDP et applications (2021--2022), Exp.~No.~XVII, 16~pp.

\bibitem{PacherieRecent}
E.~Pacherie, \emph{Unique and nonunique minimizing travelling waves for some Ginzburg--Landau type equations}, Journ\'ees \'Equations aux D\'eriv\'ees Partielles (2024), Exp.~No.~IX, 10~pp.

\bibitem{Pitaevskii}
L.~P.~Pitaevskii, \emph{Vortex lines in an imperfect Bose gas}, Soviet Phys. JETP \textbf{13} (1961), 451--454.

\bibitem{SerfatyMeanField}
S.~Serfaty, \emph{Mean field limits of the Gross--Pitaevskii and parabolic Ginzburg--Landau equations}, J. Amer. Math. Soc. \textbf{30} (2017), no.~3, 713--768.

\bibitem{SerfatyRotating}
S.~Serfaty, \emph{On a model of rotating superfluids}, ESAIM Control Optim. Calc. Var. \textbf{6} (2001), 201--238.

\bibitem{SmetsVanSchaftingen}
D.~Smets and J.~Van Schaftingen, \emph{Desingularization of vortices for the Euler equation}, Arch. Ration. Mech. Anal. \textbf{198} (2010), no.~3, 869--925.

\bibitem{Struwe}
M.~Struwe, \emph{The existence of surfaces of constant mean curvature with free boundaries}, Acta Math. \textbf{160} (1988), no.~1--2, 19--64.

\bibitem{Tarquini}
E.~Tarquini, \emph{A lower bound on the energy of travelling waves of fixed speed for the Gross--Pitaevskii equation}, Monatsh. Math. \textbf{151} (2007), no.~4, 333--339.

\bibitem{TeschlSchrodinger}
G.~Teschl, \emph{Mathematical Methods in Quantum Mechanics: With Applications to Schrödinger Operators}, 2nd ed., Graduate Studies in Mathematics, vol.~157, American Mathematical Society, Providence, RI, 2014.

\bibitem{WeiYao}
J.~Wei and W.~Yao, \emph{Asymptotic axisymmetry of the subsonic traveling waves to the Gross--Pitaevskii equation}, Comm. Contemp. Math. \textbf{13} (2011), no.~6, 1095--1104.

\end{thebibliography}
\end{document}